\documentclass{amsart}
\usepackage{amsmath, amsfonts, amssymb,color}
\usepackage{stmaryrd,epsfig}
\usepackage{xypic}

\newtheorem{thm}{Theorem}[section] %\numberwithin{thm}{chapter}
\newtheorem{lemma}[thm]{Lemma}
\newtheorem{prop}[thm]{Proposition}
\newtheorem{cor}[thm]{Corollary}

\theoremstyle{definition}
\newtheorem{remark}[thm]{Remark}

\newtheorem{defn}[thm]{Definition}

\DeclareMathOperator{\Gr}{Gr}
\DeclareMathOperator{\Fl}{Fl}

\DeclareMathOperator{\OG}{OG}
\DeclareMathOperator{\OF}{OF}

\DeclareMathOperator{\SL}{SL}
\DeclareMathOperator{\Sp}{Sp}

\DeclareMathOperator{\QK}{QK}

\renewcommand{\P}{{\mathbb P}}
\newcommand{\C}{{\mathbb C}}
\newcommand{\Z}{{\mathbb Z}}

\newcommand{\cF}{{\mathcal F}}

\newcommand{\cO}{{\mathcal O}}

\newcommand{\gw}[2]{\langle #1 \rangle^{\mbox{}}_{#2}}

\newcommand{\euler}[1]{\chi_{_{#1}}}

\newcommand{\id}{\text{id}}

\newcommand{\ev}{\operatorname{ev}}
\newcommand{\wt}{\widetilde}

\newcommand{\wb}{\overline}

\newcommand{\ignore}[1]{}

\newcommand{\Mb}{\wb{\mathcal M}}

\newcommand{\ek}{\varepsilon_k}
\newcommand{\cG}{\mathcal{G}}
\newcommand{\calH}{\mathcal{H}}
\DeclareMathOperator{\codim}{codim}
\newcommand{\EV}{\mathrm{EV}}

\begin{document}
\title[KGW invariants of lines on $G/P$]{K-theoretic Gromov-Witten invariants of lines in homogeneous spaces}
%\author{Changzheng Li and Leonardo Mihalcea}

\author{Changzheng Li}
\address{Changzheng~Li, Kavli Institute for the Physics and Mathematics of the Universe (Kavli IPMU),
Todai Institutes for Advanced Study, The University of Tokyo,
5-1-5 Kashiwa-no-Ha,Kashiwa City, Chiba 277-8583, Japan}
\email{changzheng.li@ipmu.jp}

\author{Leonardo~C.~Mihalcea}
\address{Leonardo~C.~Mihalcea, 460 McBryde Hall, Department of Mathematics, Virginia Tech,
  Blacksburg, VA 24061, USA}
\email{lmihalce@math.vt.edu}


\thanks{The first author was supported in part by a Grant-in-Aid for Young Scientists ((B)25870175) of JSPS}

\thanks{The second author was supported in part by NSA Young
  Investigator Award 12-0871-10.}

\maketitle

\begin{abstract}Let $X=G/P$ be a homogeneous space and $\ek$ the homology class of a simple coroot. For almost all $X$, the variety $Z_k(X)$ of degree $\ek$ pointed lines in $X$ is known to be homogeneous. For these $X$ we show that the $3$-point, genus $0$, equivariant K-theoretic Gromov-Witten invariants of lines of degree $\ek$ equal quantities obtained in the (ordinary) equivariant K-theory of $Z_k(X)$. We apply this to compute the Schubert structure constants $N_{u,v}^{w, \ek}$ in the equivariant quantum K-theory ring of $X$. Using geometry of spaces of lines through Schubert or Richardson varieties we prove vanishing and positivity properties of $N_{u,v}^{w,\ek}$.~This generalizes many results from K-theory and quantum cohomology of $X$, and gives new identities among the structure constants in equivariant K-theory of $X$. \end{abstract}

%\begin{abstract} Let $X=G/P$ be a homogeneous space and $\ek$ be the class of a simple coroot in $H_2(X)$. A theorem of Strickland shows that for almost all $X$, the variety of pointed lines of degree $\ek$,  denoted $Z_k(X)$, is again a homogeneous space. For these $X$ we show that the $3$-point, genus $0$, equivariant K-theoretic Gromov-Witten invariants of lines of degree $\ek$ are equal to quantities obtained in the (ordinary) equivariant K-theory of $Z_k(X)$. We apply this to compute the structure constants $N_{u,v}^{w, \ek}$ for degree $\ek$ from the multiplication of two Schubert classes in the equivariant quantum K-theory ring of $X$. Using geometry of spaces of lines through Schubert or Richardson varieties we prove vanishing and positivity properties of $N_{u,v}^{w,\ek}$.~This generalizes many results about K-theory and quantum cohomology of $X$, and also gives new identities among the structure constants in equivariant K-theory of $X$. \end{abstract}
%
%

\section{Introduction} Let $G$ be a simple, simply connected, complex Lie group, and $T \subset B  \subset G$ be a maximal torus $T$ included in a Borel subgroup $B$. Fix $P \supset B$ a parabolic subgroup of $G$ containing the Borel, and let $X=G/P$ be the associated flag manifold. We also fix a homology class $\ek \in H_2(X; \Z)$ corresponding to a simple coroot $\alpha_k^\vee$ in the dual root system associated to $G$. Strickland \cite{strickland:lines} and independently Landsberg and Manivel \cite{landsberg.manivel:projective} classified the Fano variety $L_k(X)$ of lines of class $\ek$ which are included in $X$. It turned out that for most homogeneous spaces $X$ -   notably for all simply laced groups $G$, or when $P=B$, or when $X$ is cominuscule -
the variety $L_k(X)$ is again homogeneous. Furthermore, in any of these   cases, the variety of {\em pointed lines} \[ Z_k(X) = \{ (x \in \ell): \ell \subset X \textrm{ and }\ell \textrm{ has class } \ek \} \/, \] is again homogeneous. For example, $Z_k(G/B) = G/B$; if $X=\Gr(p,m)$ is the Grassmannian of $p$-dimensional subspaces in $\C^m$, then $Z_k(\Gr(p,m)) = \Fl(p-1,p,p+1;m)$ - the three-step flag manifold parametrizing triples of subspaces $(W_{p-1} \subset W_p \subset W_{p+1})$ in $\C^m$ with $\dim W_i =i$. See \S \ref{ss:lines} for details.

The first objective of this paper is to compute the $3$-point, genus $0$, $T$-equivariant K-theoretic Gromov-Witten (KGW) invariants $\gw{[\cF_1],[\cF_2],[\cF_3]}{\ek}$ on such $X$, where $[\cF_i] \in K_T(X)$ is the class determined by the equivariant coherent sheaf $\cF_i$ on $X$. Givental \cite{givental:onwdvv} defined these invariants (for more general $d \in H_2(X; \Z)$) as the  sheaf Euler
characteristic \[ \gw{[\cF_1],[\cF_2],[\cF_3]}{\ek} := \euler{\Mb_{0,3}(X,\ek)} (\ev_1^*[\cF_1] \cdot \ev_2^*[\cF_2] \cdot \ev_3^*[\cF_3]) \] over the moduli space of stable maps $\Mb_{0,3}(X,\ek)$, where $\ev_i$ are the evaluation maps; see \S \ref{ss:KGW} below.~If $\cF_i$ is the structure sheaf of a variety $\Omega_i \subset X$ and $\sum \codim \Omega_i = \dim \Mb_{0,3}(X, \ek)$ then one recovers the cohomological GW invariants $\gw{[\Omega_1],[\Omega_2], [\Omega_3]}{\ek}$; see \cite{fulton.pandharipande}. The KGW invariants are the building blocks of $\QK_T(X)$ - the equivariant quantum K-theory algebra of $X$. This is a deformation of $K_T(X)$, defined by Givental and Lee \cite{givental:onwdvv,lee:qk}. A study of this ring and its structure constants for cominuscule Grassmanians was started in \cite{buch.m:qk,buch.chaput.ea:finiteness}, but very little is known beyond this situation. Our main objective is to study the structure constants of $\QK_T(X)$ for degree $\ek$, where $X$ is in the (large) class of homogeneous spaces described previously. These structure constants are defined with respect to the Schubert basis $\{ \cO^u \}$ where $\cO^u \in K_T(X)$ is the class of the structure sheaf of the Schubert variety $Y(u) = \overline{B^-uP/P }$, with $B^-$ the opposite Borel subgroup; $u$ is the minimal length representative in its coset in $W/W_P$ - the Weyl group of $G$ modulo that of $P$ (cf. \S \ref{ss:schubert}). There is an identity in $\QK_T(X)$ (cf. \S \ref{s:QK} below): \[ \cO^u \circ \cO^v = \sum_{z} c_{u,v}^z \cO^z +  \sum_{w, k} N_{u,v}^{w, \ek} q_k \cO^w + \textrm{ terms with higher powers of } q \/. \] Here $c_{u,v}^z \in \Lambda:=K_T(pt)$ (the representation ring of $T$) are the structure constants of $K_T(X)$ (since $\QK_T(X)$ is a deformation of it), $q_k$ are the quantum parameters, and $N_{u,v}^{w, \ek} \in \Lambda$; cf.~\S \ref{s:QK} below. An additional difficulty in computing $N_{u,v}^{w, \ek}$ is that, unlike in the ordinary quantum cohomology, they are {\em not} single KGW invariants, but an alternating sum of these.

\subsection{Statement of results} We state next a more precise version of our results. Because the formulas are simpler, in this introduction we restrict to the case when $Z_k(X) = X$. This holds if the nodes in the Dynkin diagram of $G$ which determine the parabolic group $P$ are not adjacent to node $k$. For example, $X=G/B$ satisfies this condition, but $X=\Gr(p,m)$ does not. We call these parabolics {\em $k$-free} - cf. \S \ref{ss:lines} below. Recall that $K_T(X)$ is a $\Lambda$-module with a $\Lambda$-basis consisting of the Schubert classes $\cO^u$. Another basis is given by the opposite Schubert classes $\cO_u=[\cO_{X(u)}]$, associated to Schubert varieties $X(u) = \overline{BuP/P}$. Consider the K-theoretic divided difference operator $\partial_k: K_T(X) \to K_T(X)$ associated to the root $\alpha_k$. This is a $\Lambda$-module endomorphism which satisfies: \[ \partial_k(\cO^u) = \cO^{u_k}; \quad \partial_k(\cO_u) = \cO_{u^k} \/.\] See equation (\ref{E:hecke}) below for formulas of $u_k,u^k$ and Lemma \ref{lemma:div} for proofs.

\begin{thm}\label{T:mainintro} Let $P$ be a $k$-free parabolic group, $u,v,w$ minimal length representatives in $W/W_P$ and $[\cF],[\cG],[\calH] \in K_T(G/P)$. Then

(a) The equivariant KGW invariant $\gw{[\cF],[\cG],[\calH]}{\ek}$ equals \[\gw{[\cF],[\cG],[\calH]}{\ek} = \euler{G/P}(\partial_k([\cF]) \cdot \partial_k([\cG]) \cdot [\calH]) \/, \] where $\cdot$ denotes the multiplication in $K_T(G/P)$. In particular, \[ \gw{\cO^u,\cO^v, [\cF]}{\ek}= \euler{G/P} (\cO^{u_k} \cdot \cO^{v_k} \cdot [\cF]) \/.\] (These are the relevant invariants needed to define the quantum K multiplication.)

(b) The structure constant $N_{u,v}^{w, \ek}$ in $\QK_T(G/P)$ equals the coefficient of $\cO^w$ in the expression \[\partial_k(\cO^u) \cdot \partial_k(\cO^v) - \partial_k(\cO^u \cdot \cO^v) \in K_T(G/P) \/. \] Equivalently, if $s_k$ is the reflection associated to the root $\alpha_k$ and $\delta$ is the Kronecker symbol,
\begin{equation}\label{E:QKstr} N_{u,v}^{w, \ek} = c_{u_k,v_k}^w - \delta_{w^k, ws_k} (c_{u,v}^w + c_{u,v}^{w^k}) \/. \end{equation}
%where $s_k$ is the reflection associated to the root $\alpha_k$ and $\delta$ is the Kronecker symbol.

(c) Assume that $u_k =u$ or $v_k = v$. Then $N_{u,v}^{w, \ek} = 0$ in $\QK_T(X)$.

(d) The non-equivariant structure constant $N_{u,v}^{w,\ek}$ satisfies the positivity property:

\[ (-1)^{\ell(u) + \ell(v) - \ell(w)} N_{u,v}^{w, \ek} \ge 0 \/. \]

\end{thm}
Parts (a)-(c) of the Theorem \ref{T:mainintro} generalize to homogeneous spaces $X=G/P$, where $P$ covers ``almost all" parabolic groups - see Defn. \ref{defn:P} below. In this case $Z_k(X) \neq X$ and a ``quantum" structure constant $N_{u,v}^{w,\ek}$ equals a combination of ``classical" structure constants in $K_T(Z_k(X))$; see Thms. \ref{T:mainKGW} and \ref{T:formulagenP} below for details. This ``quantum=classical" formula is in the spirit of similar formulas discovered by Buch-Kresch-Tamvakis \cite{buch.kresch.ea:gw} and extended by Chaput-Manivel-Perrin \cite{chaput.manivel.perrin} for the GW invariants on the (cominuscule) Grassmanians. This was recently generalized in \cite{buch.kresch.ea:qpr} for submaximal isotropic Grassmannians, and by N.~C.~Leung and the first author \cite{leung.li:functorial, leung.li:classical} for more general homogeneous spaces. An equivariant K-theory generalization of these results for cominuscule Grassmannians was obtained by Buch and the second author in \cite{buch.m:qk}. The vanishing property is generalized in Thm. \ref{T:vanishing} below. Combined with formula (\ref{E:QKstr}) this implies some remarkable identities among the ordinary structure constants in $K_T(X)$, which resemble the ``dc-triviality" and ``descent-cycling" conditions discovered by Knutson \cite{knutson:descent,knutson:noncomplex} in equivariant cohomology of $G/B$. A further study of such identities in the context of equivariant K Schubert Calculus should be of independent interest.

The positivity result in (d) generalizes a positivity result of Brion \cite{brion:Kpos} in the K-theory of $X$ (discovered by Buch \cite{buch:K} for Grassmannians). We conjecture - and partially prove - a generalization to the equivariant case in Rmk. \ref{remark:KTpos} below. To prove (d) we show that if $u_k \neq u$ and $v_k \neq v$ then both $c_{u_k,v_k}^w$ and $(c_{u,v}^w + c_{u,v}^{w^k})$ have the expected sign. The reason is that both are coefficients in the expansion of the classes of {\em projected Gromov-Witten varieties} in terms of Schubert classes; since the former have rational singularities, Brion's \cite[Thm. 1]{brion:Kpos} (cf. Thm. \ref{T:brion} below) implies the result. This leads us to the key technical facts needed in the proof of Thm. \ref{T:mainintro}, which we briefly explain next. Define \[\mathcal{D} = \Mb_{0, \{1,2,\bullet\}}(X, 0) \times_{X} \Mb_{0,\{\bullet, 3\}}(X,\ek) \subset \Mb_{0,3}(X,\ek) \] to be the boundary component containing a general map $f:C_1 \cup C_2 \to X$, where $C_i \simeq \P^1$, the first two marked points are on $C_1$ (which is collapsed through $f$), and the third point is on $C_2$; the map to $X$ is given by evaluating at the intersection $\{ \bullet\} = C_1 \cap C_2$. Define the {\em (boundary) Gromov-Witten varieties} $GW_{\ek}(z,v) \subset \Mb_{0,3}(X, \ek)$ respectively $GW_{0,\ek}(z,v) \subset \mathcal{D}$ to be \[ \ev_1^{-1} (X(z)) \cap \ev_2^{-1}(Y(v)) \/,\] where $\ev_i:\Mb_{0,3}(X,d) \to X$ is the evaluation map at the $i$-th marking, or its restriction to $\mathcal{D}$. The projected GW varieties are subvarieties of $X$ defined by \[ \Gamma_{\ek}(z,v) = \ev_3(GW_{\ek}(z,v)); \quad  \Gamma_{0,\ek}(z,v) = \ev_3(GW_{0,\ek}(z,v)) \/. \]
In other words, $\Gamma_{\ek}(z,v)$ is the locus consisting of the union of lines $\ell \subset X$ of degree $\ek$ which intersect the Schubert varieties $X(z)$ and $Y(v)$. Similarly, $\Gamma_{0, \ek}(z,v)$ consists of lines which intersect the Richardson variety $X(z) \cap Y(v)$. Our main technical result is the following (cf.~ Thm. \ref{T:main}, Prop. \ref{prop:coincidence} and Lemma \ref{lemma:ample} below): \begin{thm}\label{T:maintechintro} Let $P$ be a $k$-free parabolic group. Then:

(a) The projected GW variety $\Gamma_{\ek}(z,v)$ equals the Richardson variety $X(z^k) \cap Y(v_k)$.

(b) The projected (boundary) GW variety $\Gamma_{0,\ek}(z,v)$ has rational singularities, and there are inclusions \[ X(z) \cap Y(v) \subset \Gamma_{0,\ek}(z,v) \subset X(z^k) \cap Y(v_k) =\Gamma_{\ek}(z,v)\/.\] If $z^k = z$ or $v_k = v$ then $\Gamma_{0,\ek}(z,v) =  \Gamma_{\ek}(z,v)$; the inclusions are strict otherwise.

(c) The evaluation map $\ev_3:GW_{\ek}(z,v) \to \Gamma_{\ek}(z,v)$ is cohomologically trivial, i.e.  $(\ev_3)_*\cO_{GW_{\ek}(z,v)} = \cO_{\Gamma_{\ek}(z,v)}$ and the higher direct images $R^i (\ev_3)_*\cO_{GW_{\ek}(z,v)} =0$ for $i >0$. The same holds for the restriction $\ev_3:GW_{0,\ek}(z,v) \to \Gamma_{0,\ek}(z,v)$.
\end{thm}

Part (c) of the Theorem \ref{T:maintechintro} allows us to transfer a computation in the K-theory of the moduli spaces $\Mb_{0,3}(X, \ek)$ or $\mathcal{D}$ to one on $X$, while parts (a) and (b) help to compute explicitly the resulting expressions in $K_T(G/P)$. The proof of this Theorem relies heavily on the geometry of spaces of lines in $G/P$. We believe that similar results hold in a much greater generality (for all degrees $d$, and all projected GW varieties), and this is a particular instance of that phenomenon.\\

{\em Acknowledgements.} C.~L.~wishes to thank   Naichung Conan Leung and Bumsig Kim for valuable suggestions and constant encouragement during the preparation of this project.~L.~M. wishes to thank his collaborators Anders Buch, Pierre-Emmanuel Chaput and Nicolas Perrin for inspiring conversations. He is grateful to Allen Knutson for an insightful discussion about the geometry of lines on $G/B$ and to Dave Anderson for some useful remarks about the equivariant K-theory. The authors are thankful to the referees for a careful reading of this paper and valuable comments.

\section{Preliminaries} The goal of this section is to establish the notation and the basic definitions used throughout the paper.

\subsection{Lie data}\label{ss:Lie}

Let $G$ be a  simple, simply connected, complex  Lie group, and fix $T \subset B \subset G$ a Borel subgroup of $G$ containing a maximal torus $T$. Let $W = N(T)/T$ be the associated Weyl group, where $N(T)$ denotes the normalizer of the torus. Each $w \in W$ has a length $\ell(w)$; denote by $w_0$ the longest element in $W$, and by $id$ the identity. Associated to this datum one has the set of roots $R$, positive roots $R^+$, and simple roots $\Delta = \{\alpha_1, ..., \alpha_r \}$. Recall that $W$ is generated by the simple reflections $s_i=s_{\alpha_i}$, for $\alpha_i \in \Delta$. Let $(-,-)$ denote the $W$-invariant inner product on $\mathbb{R}\Delta$.
Each root $\alpha \in R$ has a coroot $\alpha^\vee = \frac{2\alpha}{(\alpha,\alpha)}$.
The coroots form the dual root system $R^\vee = \{ \alpha^\vee \mid \alpha \in R
\}$, with a basis of simple coroots $\Delta^\vee = \{ \beta^\vee \mid \beta \in
\Delta \}$.  For $\beta \in \Delta$ we let $\omega_\beta \in \mathbb{R}\Delta$
denote the corresponding fundamental weight, defined by $(\omega_\beta,
\alpha^\vee) = \delta_{\alpha,\beta}$ for $\alpha \in \Delta$.

For the parabolic subgroup $P \supset B$ we denote by $\Delta_P \subset \Delta$ the subset of simple roots in $P$, and by $W_P$ the subgroup of $W$ generated by the reflections of roots in $\Delta_P$. Let $W^P$ be the set of minimal length representatives for cosets in $W/W_P$. It can be characterized as \[ W^P = \{ w\in W: w(\alpha) > 0, \forall \alpha \in \Delta_P \} \/; \] see e.g. \cite[Ch.2,\S 5.1]{gon.lak:flagv}. Then for each coset in $W/W_P$ define $\ell(wW_P) = \ell(\overline{w})$ where $\overline{w}$ is the minimal length representative in the coset $wW_P$. Recall that there is a partial order on $W$ called the {\em Bruhat order}, which is determined by the covering relations for $u \le v$ if and only if $v= u s_\alpha$, for $\alpha \in R^+$ and $\ell(v) > \ell(u)$. This induces a partial order on $W/W_P$ by projection: $uW_P \le vW_P$ if and only if $uw \le vw'$ for some $w, w' \in W_P$; we also refer to this as Bruhat order.

Let $\alpha_k \in \Delta \setminus \Delta_P$ and $s_k$ the corresponding simple reflection. For $w \in W$ we denote by $w_k$ respectively $w^k$ the Weyl group elements
\begin{equation}\label{E:hecke} w_k= \left \{
\begin{array}{ll}
ws_{k} & \textrm{if } \ell(ws_{k}) < \ell(w) \\
w & \textrm { otherwise}
\end{array} \right., \quad w^k= \left \{
\begin{array}{ll}
ws_{k} & \textrm{if } \ell(ws_{k}) > \ell(w) \\
w & \textrm { otherwise}
\end{array} \right. . \end{equation}

\begin{defn} A parabolic subgroup $P$ is called {\em $k$-free} if $\Delta_P$ does not contain any simple root adjacent to $\alpha_k$ in the Dynkin diagram of $G$.
%For example, any Borel subgroup is $k$-free.
\end{defn}

For example, any Borel subgroup is $k$-free.

\begin{lemma}\label{lemma:minrep} Let $P$ be a $k$-free parabolic group and $w \in W^P$ a minimal length representative. Then both $w_k$ and $w^k$ are minimal length representatives in $W^P$. \end{lemma}

%\begin{lemma}\label{lemma:minrep} Assume that $w \in W^P$ is a minimal length representative, and that $\Delta_P$ does not contain any simple root adjacent to $\alpha_k$ in the Dynkin diagram of $G$. Then both $w_k$ and $w^k$ are minimal length representatives in $W^P$. \end{lemma}

\begin{proof} It suffices to show that $ws_k \in W^P$. For any $\alpha \in \Delta_P$, the hypothesis on $P$ implies that $s_k(\alpha) = \alpha$. Then $ws_k (\alpha) = w(\alpha) > 0$ since $w \in W^P$ and we are done. \end{proof}
%This finishes the proof. {\color{blue} REF humphreys , coxeter} \end{proof}

%A parabolic subgroup $P$ which satisfies the hypotheses of the Lemma will be called {\em $k$-free}. For example, any Borel subgroup $B$ is $k$-free for any $k$.

Note that $k$-freeness is necessary for the lemma to hold. For example, take $G=\SL_3(\C)$ with $\Delta= \{ \alpha_1, \alpha_2\}$, then take $P$ with $\Delta_P = \{ \alpha_2\}$. If $w = s_2 s_1$ then $w \in W^P$ but $w_1 = s_2 \notin W^P$.
We will see later that for $G/P$ where $P$ is $k$-free, the moduli space of pointed lines of degree $\ek= \alpha_k^\vee$ on $G/P$ can be identified to $G/P$ itself.
The elements $w^k, w_k$ can also be defined using the nil-Hecke, respectively the opposite nil-Hecke products on $W$. For example, $w^k = w \cdot s_k$ where $\cdot$ is the ordinary (length-increasing) nil-Hecke product on $W$ - see \cite{buch.m:nbhd}. We also notice that in our conventions $\alpha_0$ is never a root, so the index $k \neq 0$ in $\alpha_k$; thus $w_0$ is never obtained as in equation (\ref{E:hecke}) above.

\subsection{Equivariant K-theory}\label{ss:KT} Let $T \subset B$ be the maximal torus. For now let $X$ be a complex, irreducible, projective $T$-variety. Denote by $a:T \times X \to X$ the $T$-action and by $p_X: T \times X \to X$ the projection. We recall next the definition of $K_T(X)$ - the equivariant K-theory of $X$ - and its properties, following \cite[Ch. 5]{chriss.ginzburg}, \cite[\S 3.3]{brion:flagv} and \cite[\S 15.1]{fulton:intersection}. An {\em equivariant sheaf\/} on $X$ is a coherent $\cO_X$-module $\cF$ together with a given isomorphism $I : a^*\cF \cong p_X^*\cF$; this isomorphism has the property that $(m \times \id_X)^* I = p_{23}^*I \circ (\id_T
\times a)^*I$ as morphisms of sheaves on $T \times T \times X$, where
$m$ is the group operation on $T$ and $p_{23}$ is the projection to
the last two factors of $T \times T \times X$.

The {\em equivariant $K$-homology group\/} $K_T(X)$ is the
Grothendieck group of equivariant sheaves on $X$, i.e.\ the free
abelian group generated by isomorphism classes $[\cF]$ of equivariant
sheaves, modulo the relations $[\cF] = [\cF'] + [\cF'']$ if
there exists an equivariant exact sequence $0 \to \cF' \to \cF \to
\cF'' \to 0$.  This group is a module over the {\em equivariant
  $K$-cohomology ring\/} $K^T(X)$, defined as the Grothendieck group
of equivariant vector bundles on $X$.  Both the multiplicative
structure of $K^T(X)$ and the module structure of $K_T(X)$ are given
by tensor products.  If $\cF, \cG$ are equivariant sheaves there is a product in equivariant K-homology \[ [\cF] \cdot [\cG] = \sum_j (-1)^j [Tor^X_j (\cF, \cG)] \/, \] where $Tor_j^X$ is the $j$-th $Tor$ sheaf. In order for the product to be well-defined one requires that $\cF, \cG$ have finite resolutions by equivariant vector bundles. This happens if $X$ is smooth \cite[5.1.28]{chriss.ginzburg}, or if $\cF,\cG$ are pull-backs of equivariant coherent sheaves via an equivariant, flat morphism $g:Z \to X$. In particular, if $X$ is non-singular, the map $K^T(X) \to K_T(X)$ which sends a vector bundle to its sheaf of sections is an isomorphism.

%The {\em Grothendieck class\/} of $X$ is the
%class $[\cO_X] \in K^T(X)$ of its structure sheaf.

%If $X$ is
%non-singular, the map $K^T(X) \to K_T(X)$ which sends a
%vector bundle to its sheaf of sections is an isomorphism; this follows
%because every equivariant sheaf on $X$ has a finite resolution by
%equivariant vector bundles
%\cite[5.1.28]{chriss.ginzburg}.

Given an equivariant morphism of $T$-varieties $f : X \to Y$, there is
a ring homomorphism $f^* : K^T(Y) \to K^T(X)$ defined by pullback of
vector bundles. If $Y=pt$, this determines a $K^T(pt)$-module structure of $K^T(X)$, and therefore also a $K^T(pt)$-module structure on $K_T(X)$, via the module map $K^T(X) \otimes K_T(X) \to K_T(X)$. Recall that $K^T(pt) = R(T)$ - the character ring of $T$ which is a free abelian group with basis consisting of characters $e^\lambda$; set $\Lambda:=R(T)$.
%Since in our case $T$ is a maximal torus in $B$, $\Lambda:=R(T)$ can be identified to $\Z[e^{\pm \alpha_i}]_{\alpha_i \in \Delta}$.
If $f$ is proper, then there is also a pushforward
map $f_* : K_T(X) \to K_T(Y)$ defined by $f_*[\cF] = \sum_{i \geq 0} (-1)^i [R^i f_* \cF]$.  This map is a homomorphism of $K^T(Y)$-modules
by the projection formula \cite[Ex.~III.8.3]{hartshorne}. Both pullback and pushforward are functorial with respect to composition of morphisms.~Considering the (proper) morphism $\rho:X \to pt$ and $[\cF] \in K_T(X)$ we denote $\rho_*[\cF] \in \Lambda$ by $\int_X [\cF]$. Notice that \[ \int_X [\cF] = \euler{X}(\cF) = \sum_{i=0}^{\dim X} (-1)^i ch_T~H^i(X, \cF) \] is the {\em equivariant sheaf Euler characteristic} of the sheaf $\cF$, where $ch_T~M$ denotes the character of the $T$-module $M$. We will occasionally use $\euler{X}(\cF)$ instead of $\int_X [\cF]$.
% especially when we wish to emphasize the sheaf $\cF$.

Recall that the variety $X$ has {\em rational singularities\/} if
there exists a desingularization $\pi : \wt X \to X$ for which $\pi_*
\cO_{\wt X} = \cO_X$ and $R^i \pi_* \cO_{\wt X} = 0$ for $i>0$.  It turns out that if $X$ has rational singularities then it is normal, and if one desingularization satisfies the aforementioned properties then all do - see e.g. \cite{brion:flagv} and references therein. In general, we say a morphism $f:X \to Y$ is {\em cohomologically trivial~}if $f_* \cO_X = \cO_Y$ and $R^i f_* \cO_X = 0$ for $i > 0$. This will be a key property in this paper. The main tool to prove cohomological triviality is the following result, proved in \cite[Thm. 3.1]{buch.m:qk}, and which is based on a Theorem of Koll\'ar \cite{kollar:higher}:

%{\color{blue} Change fibre from rational into unirational ?}
%When
%this is true, these identities hold for any desingularization, and $X$
%is normal (see e.g the proof of
%\cite[Cor.~11.4]{hartshorne:algebraic*1}).  If $X$ is a
%$G$-variety, then it follows from
%\cite[Thm.~7.6.1]{villamayor-u:patching} or
%\cite[Thm.~13.2]{bierstone.milman:canonical} that $X$ has an {\em
%  equivariant desingularization}.  More precisely, there exists an
%equivariant projective birational morphism $\pi : \wt X \to X$ from a
%non-singular $G$-variety $\wt X$ (see e.g.\ \cite[\S
%4]{reichstein.youssin:equivariant}).  This implies that
%$\pi_*[\cO_{\wt X}] = [\cO_X] \in K_G(X)$.  More generally, if $f : X
%\to Y$ is any equivariant proper birational map of $G$-varieties with
%rational singularities, then the composition $f \pi : \wt X \to Y$ is
%a desingularization of $Y$, so we obtain $f_* [\cO_X] = f_* \pi_*
%[\cO_{\wt X}] = [\cO_Y] \in K_G(Y)$ by functoriality.  We need the
%following generalization.

\begin{thm} \label{T:push}
  Let $f : X \to Y$ be a surjective equivariant map of projective
  $T$-varieties with rational singularities.  Assume that the general
  fiber of $f$ is rational, i.e.\ $f^{-1}(y)$ is an irreducible
  rational variety for all closed points in a dense open subset of
  $Y$.  Then $f_* \cO_X = \cO_Y$ and $R^i f_* \cO_X = 0$ for $i > 0$. In particular, $f_*[\cO_X] = [\cO_Y] \in K_T(Y)$.
\end{thm}

\subsection{Schubert classes and $K_T(G/P)$}\label{ss:schubert} Consider now the flag variety $X = G/P$, endowed with the $T$-action obtained by restricting the $G$-action defined by left multiplication. For each $w \in W^P$ there are the {\em Schubert cells} $X(w)^o = BwP/P$ and $Y(w)^o = B^-wP/P$, and their closures, the {\em Schubert varieties} $X(w)= \overline{BwP/P}$ and $Y(w) = \overline{B^- wP/P}$ where $B^- = w_0 B w_0$ is the opposite Borel subgroup. Note that $\dim X(w) = \textrm{codim}~Y(w) = \ell(w)$ and that $X(w) \cap Y(w)$ consists of a unique, $T$-fixed point which we denote by $e_w$. The Schubert varieties have $T$-equivariant sheaves of regular functions $\cO_{X(w)}, \cO_{Y(w)}$ which in turn determine {\em Schubert classes} $\cO_w = [\cO_{X(w)}]$ and $\cO^w = [\cO_{Y(w)}]$ in $K^T(X)=K_T(X)$ - the $T$-equivariant K-theory of $X$. It is well known that the set of Schubert classes $\{ \cO^w \}_{w \in W^P}$ and $\{\cO_w\}_{w \in W^P}$ both form a $\Lambda$-basis of $K_T(X)$. For this and other basic facts about $K_T(X)$ which will be mentioned below see e.g. \cite{graham.kumar:positivity}.

The push-forward to a point determines a non-degenerate $\Lambda$-pairing \[ <[E], [E']> = \int_X E \otimes E' \in \Lambda; \quad [E],[E'] \in K^T(X) \/. \] The dual $(\cO^w)^\vee$ of $\cO^w$ with respect to this pairing is the class $\xi_w := [\cO_{X(w)} (-\partial X(w))]$, where $\partial X(w) = X(w) \setminus X(w)^o$ is the boundary divisor of $X(w)$. Using this pairing one can define the structure constants $c_{u,v}^w \in \Lambda$ of the equivariant K-theory ring by \[ \cO^u \cdot \cO^v = \sum_w c_{u,v}^w \cO^w; \quad c_{u,v}^w = <\cO^u \cdot \cO^v , \xi_w > = \int_X \cO^u \cdot \cO^v \cdot \xi_w \/. \] The structure constants $c_{u,v}^w$ have been heavily studied, and they will be the building blocks required to compute the structure constants of (equivariant) quantum K-theory - see \S \ref{s:QK} below. Explicit formulas and algorithms to compute these coefficients can be found e.g. in \cite{kostant.kumar:KT,lenart.postnikov:KT}.

If $P \subset Q $ are two parabolic subgroups containing $T$, there is a natural $G$-equivariant projection map $\pi:G/P \to G/Q$. The inclusion of the parabolic subgroups determines an inclusion of groups $W_P \subset W_Q$, and therefore a set inclusion $W^Q \subset W^P$. Note also that if $\Omega \subset G/P$ and $\Omega' \subset G/Q$ are Schubert varieties (either $B$ or $B^-$-stable), then $\pi(\Omega)$ and $\pi^{-1}(\Omega')$ are also Schubert varieties, and it is easy to find the associated Weyl group elements. The situations we will encounter most often are $\pi^{-1}(Y(w)) = Y(w)$ and $\pi(X(w)) = X(\overline{w})$ where $\overline{w} \in W^Q$ is the minimal length representative for the coset $wW_Q$. The projection $\pi$ induces an injection $\pi^*: K^T(G/Q) \to K^T(G/P)$ and since $\pi$ is flat $\pi^* [\cO_{\Omega'}] = [\cO_{\pi^{-1} (\Omega')}]$. Using e.g. Frobenius splitting arguments one can show that $\pi: \Omega \to \pi(\Omega)$ is cohomologically trivial, therefore $\pi_*[\cO_\Omega] = [\cO_{\pi(\Omega)}]$ - cf. \cite[Thm. 3.3.4]{brion.kumar:frobenius}. We will need a generalization of this from Schubert to {\em Richardson varieties} $R_u^v := X(u) \cap Y(v)$. This is nonempty exactly when $v \le u$ in the Bruhat order. In this case $R_u^v$ is an irreducible $T$-variety of dimension $\ell(u) - \ell(v)$; see \cite{brion:Kpos} for more details.

\begin{lemma}\label{lemma:projrich} Let $P \subset Q$ be two parabolic subgroups and $\pi:G/P \to G/Q$ the natural projection. Let $R \subset G/P$ be a non-empty Richardson variety. Then both $R$ and $\pi(R)$ have rational singularities and the restriction morphism $\pi:R \to \pi(R)$ is cohomologically trivial. \end{lemma}

\begin{proof} The fact that $R$ has rational singularities is proved in \cite[\S 1]{brion:Kpos}. The other assertions are proved in \cite{knutson-lam-speyer:projections} or \cite{billey.coskun:richardson}.\end{proof}

\subsection{K-theoretic Gromov-Witten invariants}\label{ss:KGW} We continue to use $X$ for $G/P$. The homology group $H_2(X; \Z)$ is isomorphic to $\oplus_{\alpha_i \in \Delta} \Z \alpha_i^\vee / \oplus_{\beta_j \in \Delta_P} \Z \beta_j^\vee $. A {\em degree} is an effective element $d \in H_2(X; \Z)$ and it can be written as a non-negative combination of simple coroots $d= \sum_{\alpha_j \in \Delta \setminus \Delta_P} n_j \alpha_j^\vee$. The degree $\alpha_k^\vee$ will be denoted by $\ek$. For the degree $d \in H_2(X)$ denote by $\Mb_{0, n}(X,d)$ the moduli space of (genus $0$) stable maps to $X$, which compactifies the space of rational curves of degree $d$ in $X$ with $n$ marked points. We list below some of the well-known properties of this moduli space - we refer to \cite{fulton.pandharipande} for details. We remark that in this paper we will only consider the case when $d = \ek$, and that in this special case most of these properties can be easily derived from an alternate description of the moduli space given in the next section. We leave this derivation to the interested reader and instead point to the relevant references for the general properties.

The elements of $\Mb_{0,n}(X,d)$ are equivalence classes of morphisms $f:(C; pt_1, ..., pt_n) \to X$ where $C= \bigcup C_i$ is a tree of $\P^1$'s containing the points $pt_k \in C$ which are smooth points of $C$, $f_*[C] = d$, and $f$ is stable, i.e. each component $C_i$ of $C$ such that $f(C) = pt$ contains at least three markings; a marking is either a marked point $pt_j $ or a point of intersection of two components. Corresponding to the marked points there are evaluation maps $\EV= (\ev_1, ..., \ev_n):\Mb_{0,n}(X,d) \to X^n$ sending $f:(C; pt_1, ..., pt_n) \to X$ to $(f(pt_1), ..., f(pt_n))$. The $G$-action on $X$ extends to one on the moduli space by $(g \cdot f)(x) = g \cdot f(x)$. Since the moduli space is irreducible \cite{thomsen:irreducibility} it follows that $\ev_i$ is flat. The moduli space is rational \cite{kim.pandharipande} and it has finite quotient singularities \cite{fulton.pandharipande}, therefore rational singularities.

%Combined with Theorem \ref{T:push} and the fact that $\ev_i$ is a locally trivial fibration \cite[Prop. 2.3]{buch.chaput.ea:finiteness} this implies that $\ev_i$ is a cohomologically trivial morphism.

Let $\Omega_i$, $1 \le i \le 3$ be three $T$-stable subvarieties of $X$. The ($3$-point, genus $0$, equivariant) {\em K-theoretic Gromov-Witten invariant} (KGW) is the sheaf Euler characteristic \[ \gw{[\cO_{\Omega_1}], [\cO_{\Omega_2}], [\cO_{\Omega_3}]}{d} = \euler{\Mb_{0,3}(X,d)}(\EV^*([\cO_{\Omega_1}] \times [\cO_{\Omega_2}] \times [\cO_{\Omega_3}] ))\/. \] If $\sum_{i=1}^{3} \codim_X \Omega_i = \dim \Mb_{0,3}(X,d)$ this is the ordinary Gromov-Witten invariant; the K-theoretic version was defined by Givental \cite{givental:onwdvv}.

Let $u,v \in W^P$. We will heavily use the ($2$-point) {\em Gromov-Witten varieties} \[ GW_d(u,v) = \ev_1^{-1} X(u) \cap \ev_2^{-1} Y(v) \subset \Mb_{0,3}(X,d) \/. \] If non-empty, the variety $GW_d(u,v)$ irreducible, unirational and it has rational singularities \cite[\S 3]{buch.chaput.ea:finiteness}. Denote also by \[ \Gamma_d(u,v) = \ev_3(GW_d(u,v)) \subset X \] the {\em projected Gromov-Witten variety}. The interest in these varieties comes from the fact that for any $[\cF] \in K_T(X)$, \begin{equation}\label{E:2GW} \begin{split} \gw{\cO_u, \cO^v, [\cF]}{d} = \int_{\Mb_{0,3}(X,d)} (\ev_1 \times \ev_2)^* (\cO_u \times \cO^v) \cdot \ev_3^* [\cF] = & \\ \int_{\Mb_{0,3}(X,d)} [\cO_{GW_d(u,v)}] \cdot \ev_3^*[\cF] = \int_X (\ev_3)_*[\cO_{GW_d(u,v)}] \cdot [\cF] \/; \end{split}\end{equation} here we used the projection formula and the fact that $(\ev_1 \times \ev_2)^* (\cO_u \times \cO^v) = [\cO_{GW_d(u,v)}] $. The latter follows from Sierra's K-theoretic version of Kleiman transversality Theorem \cite{sierra} and it actually holds for $n$-point GW varieties - the details of the proof are in \cite[\S 4.1]{buch.m:qk}. This computation shows that one can reduce the computation of a KGW invariant to a computation in the K-theory of $X$, provided that one can compute explicitly the push-forward $(\ev_3)_*[\cO_{GW_d(u,v)}]$. This will be done in Thm. \ref{T:main} below.

%The key technical result in this paper is that for $d = \ek$ and for a large set of parabolic groups (cf. Def. \ref{defn:P} below), the map \[ \ev_3: GW_d(u,v) \to \Gamma_d(u,v) \] is cohomologically trivial, therefore $(\ev_3)_*[\cO_{GW_d(u,v)}] = [\cO_{\Gamma_d(u,v)}]$. Moreover, in all these cases $\Gamma_d(u,v)$ is either a Richardson variety, or a {\em projected Richardson variety} - see {\color{blue} \S ? for precise definitions}.

\section{Lines in $G/P$}\label{ss:lines} The goal of this section is to introduce the variety of lines $L_k(X)$ of a fixed degree $\ek= \alpha_k^\vee \in H_2(X;\Z)$ in a homogenous space $X=G/P$. These varieties are characterized by Strickland in \cite{strickland:lines} and Landsberg and Manivel in \cite{landsberg.manivel:projective} and they can be divided in two categories. In the first category, which corresponds to most parabolic subgroups $P$, are included all homogeneous spaces $X$ so that $Z_k(X)$ is an auxiliary homogeneous space $G/Q$. But it is not always the case that $Z_k(X)$ is homogeneous, and Strickland performs a case-by-case study to identify explicitly the remaining varieties $Z_k(X)$. Our paper is concerned with the ``regular" homogeneous spaces from the first category.
% we plan to study the KGW invariants of lines included in those $X$ from the second category elsewhere.

Let $L$ be a very ample line bundle on $X$. This is determined by a weight $\lambda = \sum_i n_i \omega_i$, where $n_i=(\lambda,\alpha_i^\vee) > 0$ for all $\alpha_i \in \Delta \setminus \Delta_P$.  A {\em line} is a subvariety $\ell \subset X$ such that its image $\iota(\ell)$ under the embedding $\iota: X \subset \P(H^0(X, L)^*) := \P(V)$ is a line in $\P(V)$. Strickland shows that $\P(V)$ contains a line in $X$ if the weight $\lambda$ satisfies $(\lambda,\alpha_k^\vee) =1$ for $\alpha_k \in \Delta \setminus \Delta_P$. Moreover, such a line $\ell$ has homology class $[\ell]= \alpha_k^\vee \in H_2(X; \Z)$. Consider now $L(X)$ - the Fano variety of lines in $\P(V)$ included in $X$. This is a projective subvariety of $X$, and there is a locally constant morphism $e: L(X) \to H_2(X;\Z)$ defined by sending $\ell$ to its fundamental class $[\ell]$. Then define $L_k(X) := e^{-1} (\alpha_k^\vee)$ as a closed subset of $X$, considered as a scheme with its reduced structure. The variety $L_k(X)$ is independent of the choice of $\lambda$, and its points are complex curves in $X$ of fixed homology class $\ek$; see {\em loc.~cit.} for complete details. Define the incidence variety $Z_k(X) = \{ (x, \ell) \in X \times L_k(X): x \in \ell \}$.

\begin{defn}\label{defn:P} Consider the set $\mathcal{P}$ of pairs $(P, \alpha_k)$ of a parabolic group $P$ and a root $\alpha_k \in \Delta \setminus \Delta_P$ such that: \begin{itemize} \item either $\alpha_k$ is a long root, or \item the connected component containing $\alpha_k$ in the Dynkin diagram of $G$ consisting of roots in $\Delta_P \cup \{ \alpha_k\}$ is simply laced. \end{itemize} \end{defn}

Obviously, the set $\mathcal{P}$ contains all pairs $(P,\alpha_k)$ for simply laced groups $G$, and also all pairs where $P=B$ is a Borel subgroup (since $\Delta_B = \emptyset$). It also contains pairs $(P, \alpha_k)$ where $P$ is $k$-free, or when $P$ is a maximal parabolic subgroup, and $\alpha_k$ a cominuscule root (i.e. $\alpha_k$ appears with coefficient $1$ in the expansion of the highest root in $R^+$).

Fix $(P, \alpha_k) \in \mathcal{P}$ and define two parabolic subgroups as follows: $P_k \subset P$ is the parabolic subgroup determined by \[ \Delta_{P_k} := \Delta_P \setminus \{\alpha_i: (\alpha_i, \alpha_k^\vee) \neq 0 \} \/.\] In other words, to obtain $\Delta_{P_k}$ we remove from $\Delta_P$ the roots in the Dynkin diagram of $G$ which are adjacent to $\alpha_k$. Clearly $P_k$ is $k$-free in the sense of \S \ref{ss:Lie}, and if $P$ is already $k$-free then $P_k = P$. The second group, denoted $P(k)$, is defined by $\Delta_{P(k)} = \Delta_{P_k} \cup \{\alpha_k\}$. Recall the classification results in \cite{strickland:lines, landsberg.manivel:projective}:

\begin{thm}[\cite{strickland:lines}, Thm.~1,\cite{landsberg.manivel:projective} Thm. 4.3]\label{thm:strickland} Let $(P,\alpha_k) \in \mathcal{P}$. Then:\begin{enumerate}

\item There are natural isomorphisms $L_k(G/P) \simeq G/P(k)$ and $Z_k(G/P) \simeq G/P_k$.

%\item The incidence variety $Z_k(G/P)$ has a natural identification with $G/P_k$.

\item  The previous isomorphisms are compatible with the natural projections, i.e there is a commutative diagram

\[ \xymatrix{Z_k(G/P) \ar[r]^\simeq \ar[d]^{pr_2} & G/P_k  \ar[d]^\pi \\ L_k(G/P) \ar[r]^\simeq & G/P(k) } \]

%Under the identification in (2), the projection $\pi_1 (\ell, x) = \ell$ can be identified with the projection $G/P_{\overline{S}_k} \to G/P_{S_k}$ and the projection $\pi_2(\ell, x) = x$ with the projection $G/P_{\overline{S}_k} \to G/P_S$.

\end{enumerate}

\end{thm}

The proof of the Theorem uses the natural $G$ action on $L_k(G/P)$ and $Z_k(G/P)$. One first proves that $G$ acts transitively on both varieties, then identifies the stabilizer of a $T$-fixed point in both. The moduli spaces $\Mb_{0,0}(G/P, \ek)$ and $\Mb_{0,1}(G/P, \ek)$ also admit a natural $G$-action and the moduli points can be identified with lines, respectively pointed lines in $G/P$. (For example, if $f:C \to X$ is a point in $\Mb_{0,0}(X,\ek)$, the stability condition implies that $C \simeq \P^1$, and the equivalence class of $f$
corresponds to reparametrizations.) This shows: \begin{cor}\label{cor:identification} There are natural isomorphisms $\Mb_{0,0}(G/P, \ek) \simeq G/P(k)$ and $\Mb_{0,1}(G/P, \ek) \simeq G/P_k$, and a commutative diagram as in Thm. \ref{thm:strickland} obtained by replacing the varieties $L_k(G/P)$ and $Z_k(G/P)$ by the appropriate moduli spaces; the evaluation map $\ev_1$ corresponds to $pr_2$. \end{cor}

%\begin{remark} The Corollary implies that if $P$ is $k$-free then $\Mb_{0,1}(G/P)$ can be identified to $G/P$. We will re-obtain

\section{(Projected) Gromov-Witten varieties and cohomological trivial maps}

The goal of this section is to prove the technical results needed to make explicit calculations of KGW invariants. We first introduce some notations and the general setup. Let $(P,\alpha_k) \in \mathcal{P}$ and $Q$ a $k$-free parabolic group satisfying $B \subset Q \subset P$. The natural projection $\pi:G/Q \to G/P$ induces a map $\Pi:\Mb_{0,3}(G/Q, \ek) \to \Mb_{0,3}(G/P, \ek)$ by $\Pi(f) = \pi\circ f$. Fix $u, v \in W^P$ and let $\widehat{u} \in W^Q$ be defined by $\pi^{-1}(X(u)) = X(\widehat{u})$. Clearly $\Pi(GW_{\ek}(\widehat{u},v)) \subset GW_{\ek}(u,v)$ and abusing notation we denote the restricted map again with $\Pi$. There is a commutative diagram:
 \begin{equation}\label{E:diagram} \xymatrix{ GW_{\ek}(\widehat{u},v) \subset \Mb_{0,3}(G/Q,\ek) \ar[r]^{\Pi} \ar[d]^{\ev_3} &  GW_{\ek}(u,v) \subset \Mb_{0,3}(G/P,\ek)\ar[d]^{\ev_3} \\ \Gamma_{\ek}(\widehat{u},v) \subset G/Q \ar[r]^{\pi} & \Gamma_{\ek}(u,v)\subset G/P }\end{equation} where the bottom map is the restriction of $\pi$. The main theorem of this section is:

\begin{thm}\label{T:main} Assume that $GW_{\ek}(u,v)$ is non-empty.

(a) If $P$ is $k$-free, the projected Gromov-Witten variety $\Gamma_{\ek}(u,v)$ equals the Richardson variety $X(u^k) \cap Y(v_k)$.

(b) For any $(P,\alpha_k) \in \mathcal{P}$,  $GW_{\ek}(\widehat{u},v) = \Pi^{-1} (GW_{\ek}(u,v))$.

(c) If $(P, \alpha_k) \in \mathcal{P}$ then all varieties in diagram (\ref{E:diagram}) have rational singularities, and all maps are surjective and cohomologically trivial.
\end{thm}

In the situation when $Q=B$, the claimed properties of $\Pi$ can also be derived from more general results of Woodward \cite{woodward:peterson} proving Peterson's comparison formula. The idea is that for homogeneous spaces $G/P$ where $(P,\alpha_k) \in \mathcal{P}$, the Peterson lift of the degree $\ek \in H_2(G/P)$ remains $\ek \in H_2(G/B)$. Our proof is different and elementary, and it relies on the geometry of spaces of lines in flag manifolds. The key part for calculations of KGW invariants is the cohomological triviality of $\ev_3$.
The proof of the theorem is divided into two main cases, for $P$ being $k$- free and $(P, \alpha_k) \in \mathcal{P}$, but not necessarily $k$-free.

\subsection{The case when $P$ is $k$-free}\label{ss:kfree} The key property satisfied by homogeneous spaces $X=G/P$ when $P$ is $k$-free is that $\Mb_{0,1}(G/P, \ek) \simeq G/P$. Although this follows from Cor. \ref{cor:identification}, we will reprove it here, in a different way. The ingredients in our proof will be used repeatedly throughout the paper. The main observation is that through any point $x \in G/P$, there exists a unique line $\ell \ni x$ of degree $\ek$. We first introduce some necessary definitions.

The {\em curve neighborhood} $\Gamma_d(\Omega)$ of a subvariety $\Omega \subset X$ is the locus of points $x \in X$ so that there exists a rational curve $C$ of degree $d$ so that $x \in C$ and $C \cap \Omega \neq \emptyset$. The set $\Gamma_d(\Omega)$ can also be realized as $\ev_2(\ev_1^{-1} \Omega)$, and this gives it a scheme structure. It follows that if $\Omega$ is a $B$-stable variety then so is its curve neighborhood, so it must be a union of $B$-stable Schubert varieties. In fact, if $\Omega$ is a $B$-stable Schubert variety, then so is $\Gamma_d(\Omega)$ \cite[Prop. 3.2]{buch.chaput.ea:finiteness}. The Weyl group element corresponding to this variety was identified in \cite{buch.m:nbhd}. For the convenience of the reader, we include this proof in the case when $P$ is $k$-free and $d= \ek$, when the arguments are simpler. From now on in this section $P$ is a $k$-free parabolic subgroup.

\begin{prop}\label{prop:kfreelines} Let $P$ be a $k$-free parabolic subgroup, and $u \in W^P$. Then $\Gamma_{\ek}(X(u)) = X(u^k)$ and $\Gamma_{\ek}(Y(u)) = Y(u_k)$. \end{prop}
\begin{proof} We will only show the formula for $\Gamma_{\ek}(X(u))$ - the other formula is similar. It is clear that $X(u^k) \subset \Gamma_{\ek}(X(u))$.
%Let $x \in \Gamma_{\ek} (X(u))$ and $\ell \ni x$ a line of degree $\ek$ with $\ell \cap X(u) \neq \emptyset$. There exists $b \in B$ so that $b \cdot \ell \cap X(u)$ contains a $T$-fixed point $e_w \in X(u)$. %W.l.o.g. we can assume $w=u$; otherwise $w \le u$ in Bruhat ordering, and $\Gamma_{\ek} (X(w)) \subset \Gamma_{\ek}(X(u))$. We show next that there exist a $T$-stable line passing through $e_u$.
Consider the Gromov-Witten variety $GW_{\ek}(e_v) := \ev_3^{-1} (e_v) \subset \Mb_{0,3}(X,\ek)$ where $e_v \in X(u)$ is a $T$-fixed point; this variety is non-empty because $\ev_3$ is surjective. Since $T$ is a connected solvable algebraic group acting on $GW_{\ek}(e_v)$, it follows that it contains $T$-fixed point \cite[\S 21.2]{humphreys:linalg}. This corresponds to a morphism $f:(C; pt_1, pt_2,pt_3) \to X$ and the image of this
morphism is a $T$-stable line of degree $\ek$, containing $e_v$. We denote this line with $\ell$, and notice that $\ell \subset \Gamma_{\ek}(X(v))$. Further, because $X(v) \subset X(u)$ we have that $\Gamma_{\ek}(X(v)) \subset \Gamma_{\ek}(X(u))$ so we can assume that $v=u$.

It is well-known (cf.~e.g.~\cite[\S 2]{carrell.kuttler:onTvars} or \cite{fulton.woodward:on}) that any irreducible $T-$stable curve $C$ in $X$ contains exactly two $T$-fixed points of the form $e_{uW_P}$ and $e_{us_\alpha W_P}$, where $\alpha \in R^+ \setminus R_P^+$, and that this curve has degree $\alpha^\vee \in H_2(X)$. Therefore $\ell$ joins $e_u$ to $e_w$, where $wW_P = us_\alpha W_P$ for a root $\alpha \in R^+ \setminus R_P^+$ with \begin{equation}\label{E:1} \alpha^\vee  - \alpha_k^\vee \in \oplus_{\alpha_j \in \Delta_P} \Z \alpha_j^\vee \/. \end{equation}
If $\alpha \neq \alpha_k$ then $\alpha^\vee$ must contain in its decomposition at least one simple coroot $\alpha_i^\vee$ so that nodes $i$ and $k$ are adjacent in the Dynkin diagram of $G$. Since $P$ is $k$-free, $\alpha_i \in \Delta \setminus \Delta_P$, which contradicts equation (\ref{E:1}). Then $\alpha= \alpha_k$ and $w = us_k$, which implies that $\Gamma_{\ek}(X(u)) \subset X(u^k)$. This finishes the proof. \end{proof}

\begin{cor}\label{cor:kfreelines} For any $x \in X$, there exists a unique line $\ell$ of degree $\ek$ which contains $x$. In particular, this line is isomorphic to a $G$ translate of the Schubert variety $X(s_k)$.  \end{cor}

\begin{proof} Without loss of generality, one can assume that $x=e_{id}$ is the unique $B$-fixed point in $X$. Then apply the previous proposition to $X(u) = X({id})$.\end{proof}

\begin{cor}\label{cor:ident} The evaluation map $\ev_1: \Mb_{0,1}(X,\ek) \to X$ is an isomorphism. \end{cor}
\begin{proof} Cor. \ref{cor:kfreelines} shows that the set theoretic fibre over each point in $X$ consists of a single element. Since both varieties are normal, it follows that $\ev_1$ is an isomorphism \cite[p.~209]{mumford:redbook}. \end{proof}

\begin{remark} The proof of Cor. \ref{cor:ident} did not use the characterization Thm. \ref{thm:strickland}, and related techniques can be used to give an alternative proof of this theorem. We do not use this anywhere else in the paper thus we only sketch the idea of the proof. Consider the natural map $\Pi:\Mb_{0,1}(G/P_k, \ek) \to \Mb_{0,1}(G/P, \ek)$ induced by the projection $\pi:G/P_k \to G/P$. Using results from \cite{buch.m:nbhd} one can calculate that for $(P, \alpha_k) \in \mathcal{P}$ the curve neighborhood $\Gamma_{\ek}(\pi^{-1}(1.P)) = X(w_P w_{P_k} s_k W_{P_k})$ where $w_P,w_{P_k}$ are the longest elements in $W_P$ and $W_{P_k}$. Further, $w_P w_{P_k} s_k$ is the minimal length representative in both its $W_{P_k}$ and $W_P$ cosets. Then the restriction \[ \pi: \Gamma_{\ek}(\pi^{-1}(1.P)) = X(w_P w_{P_k} s_k W_{P_k}) \to \Gamma_{\ek}(1.P) = X(w_P w_{P_k} s_k W_P)\] is birational, and thus so is $\Pi$. But $ \Mb_{0,1}(G/P_k, \ek) \simeq G/P_k$ since $P_k$ is $k$-free. Because $\Pi$ is $G$-equivariant it follows that $G$ acts transitively on $\Mb_{0,1}(G/P, \ek)$, therefore $\Pi$ must be an isomorphism. This recovers Cor. \ref{cor:identification} and ultimately Thm.~\ref{thm:strickland}.\end{remark}

\begin{cor}\label{cor:gwimage} Let $u,v \in W^P$ for $P$ $k$-free. Then $\Gamma_{\ek}(u,v)$ equals the Richardson variety $X(u^k) \cap Y(v_k)$. \end{cor}

\begin{proof} By definition and using Prop. \ref{prop:kfreelines} \[ \Gamma_{\ek}(u,v)=\ev_3(GW_{\ek}(u,v)) \subset \ev_3(\ev_1^{-1} X(u)) \cap \ev_3(\ev_2^{-1} Y(v)) = X(u^k) \cap Y(v_k) \/. \]
We show the reverse inclusion. Let $x \in X(u^k) \cap Y(v_k) $. There exist two lines $\ell_1 \ni x$, $\ell_2 \ni x$ of degree $\ek$ so that $\ell_1 \cap X(u) \neq \emptyset$ and $\ell_2  \cap Y(v) \neq \emptyset$. But Cor. \ref{cor:kfreelines} implies that $\ell_1 = \ell_2$ and we are done.  \end{proof}

To prove that $\ev_3:GW_{\ek}(u,v) \to \Gamma_{\ek}(u,v)$ is cohomologically trivial we need one more lemma:

\begin{lemma}\label{lemma:intersection} Let $P$ be $k$-free and $\ell$ a line of degree $\ek$. Let $u \in W^P$ and assume that $\ell \cap X(u) \neq \emptyset$ and that $\ell \cap X(u^k)^o \neq \emptyset$. Then either $\ell \subset X(u)$ or the intersection $\ell \cap X(u)$ is transversal and it consists of a single point. In the first case $u^k =u$. \end{lemma}

\begin{proof} As before $X=G/P$. Consider the evaluation map \[ \ev_2:\ev_1^{-1} (X(u))\subset \Mb_{0,2}(X, \ek) \to \ev_2(\ev_1^{-1} X(u))= X(u^k) \/. \] By \cite[Prop. 3.2]{buch.chaput.ea:finiteness} and because the moduli space $\Mb_{0,2}(X, \ek)$ is irreducible \cite{thomsen:irreducibility}, this is a locally trivial fibration over the open cell $X(u^k)^o$ and it has irreducible fibres over this cell.
We claim that each of these fibres is isomorphic to $\ell \cap X(u)$. Indeed, by Prop. \ref{prop:kfreelines} the line $\ell$ satisfies that $\ell \subset X(u^k)$, and by hypothesis there exists $x_0 \in \ell \cap X(u^k)^o$. The fibre over $x_0$ consists of all lines $\ell'$ of degree $\ek$ with two marked points $x_0,y \in \ell'$ with the additional property that $y \in \ell' \cap X(u)$. (It is possible that $x_0=y$ in which case $(x_0,y,\ell')$ corresponds to an element $f:C_1 \cup C_2 \to X$ in the boundary of $\Mb_{0,2}(X, \ek)$, where $f(C_1) = x_0=y$, and $f(C_2) = \ell'$.) Cor. \ref{cor:kfreelines} implies that there exists a unique line of degree $\ek$ through $x_0$, so $\ell'=\ell$, and the fibre above is the intersection $\ell \cap X(u)$. Since $\ev_1^{-1}(X(u))$ has rational singularities \cite[Cor. 3.1]{buch.chaput.ea:finiteness}, so does a general fibre of $\ev_2$ \cite[Lemma 3]{brion:Kpos}. Using that $\ev_2$ is $B$-equivariant this implies that all fibres over $X(u^k)^o$ are isomorphic. But these fibres have dimension at most $1$, therefore they must be smooth.

Then we have two possibilities: either $\dim \ell \cap X(u) =1$, when $\ell \subset X(u)$ (and {\em a fortiori} $u^k=u$), or $\ell \cap X(u)$ is $0$-dimensional, irreducible, and normal, therefore it is a reduced point.\end{proof}
%transversal, and it consists of a single point. \end{proof}

We can now prove a part of Thm. \ref{T:main}:

\begin{thm}\label{T:maincohtriv} Let $P$ be a $k$-free parabolic subgroup and $u,v \in W^P$. Then $\Gamma_{\ek}(u,v) = X(u^k) \cap Y(v_k)$ and the map $\ev_3:GW_{\ek}(u,v) \to \Gamma_{\ek}(u,v)$ is cohomologically trivial.\end{thm}

\begin{proof} By Cor. \ref{cor:gwimage} we only need to show the cohomological triviality. The idea is to use Thm. \ref{T:push}, so we need to check that all hypotheses are satisfied. First, the Gromov-Witten variety
$GW_{\ek}(u,v)$ and the Richardson variety $\Gamma_{\ek}(u,v)$ have rational singularities, by \cite[Cor. 3.1]{buch.chaput.ea:finiteness} respectively \cite[\S 1]{brion:Kpos}.
Take $x \in X(u^k) ^o \cap Y(v_k)^o$ and let $\ell \ni x$ be the corresponding line of degree $\ek$ given by Cor. \ref{cor:kfreelines}. Note that $X(u^k) ^o \cap Y(v_k)^o$ is open and dense in $\Gamma_{\ek}(u,v)$, therefore the fibre $F_x$ over $x$ is general. We now identify this fibre, using the definition of stable maps. According to Lemma \ref{lemma:intersection}, we have three situations:   $\ell \cap X(u) = pt$ or $\ell \cap Y(v) = pt$ or $\ell \subset X(u) \cap Y(v)$. In the first two situations $F_x$ is isomorphic to $(\ell \cap X(u) ) \times (\ell \cap Y(v))$, which is a rational variety. Let now $\ell \subset X(u) \cap Y(v)$. Then $u^k = u$ and $v_k = v$. Consider the moduli space $\Mb_{0,3}(\ell, 1) \simeq \Mb_{0,3}(\P^1, 1)$, where $1 \in H_2(\P^1)$ is the fundamental class of $\P^1$. Then $F_x$ is isomorphic to $\ev_3^{-1}(x) \subset \Mb_{0,3}(\ell, 1)$. On one side $\ev_3^{-1}(x)$ contains an open dense set birational to $\ell \times \ell$. On the other side $\ev_3^{-1}(x)$ is irreducible \cite[Prop. 3.2]{buch.chaput.ea:finiteness}, therefore $F_x$ is rational. \end{proof}

\subsection{The case when $(P, \alpha_k) \in \mathcal{P}$} In this section we will prove the remaining parts of Thm. \ref{T:main}. Note that if $(P,\alpha_k) \in \mathcal{P}$, $u \in W^P$ and $Q \subset P$ then clearly $(Q,\alpha_k) \in \mathcal{P}$ and $u \in W^Q$. The key result in this section is the following:

\begin{thm}\label{T:surj} Let $(P,\alpha_k) \in \mathcal{P}$ and $Q \subset P$ a $k$-free parabolic subgroup. Then the map $\Pi: \Mb_{0,3}(G/Q, \ek) \to \Mb_{0,3}(G/P, \ek)$ is surjective and cohomologically trivial. \end{thm}

\begin{proof} We use again Thm.~\ref{T:push}, so we need to check that all hypotheses hold. Since the moduli spaces have rational singularities, it suffices to prove that $\Pi$ is surjective and a general fibre is irreducible and rational. We have a commutative diagram \[ \xymatrix{\Mb_{0,3}(G/Q, \ek) \ar[r]^{\Pi}\ar[d]^{for} & \Mb_{0,3}(G/P, \ek) \ar[d]^{for} \\ \Mb_{0,1}(G/Q, \ek) \simeq G/Q \ar[r]^{\Pi'} & \Mb_{0,1}(G/P, \ek) \simeq G/P_k }
\] where the horizontal maps are induced by the projection $\pi$, and the vertical maps are the forgetful maps, forgetting points $pt_1, pt_2$ (see \cite{fulton.pandharipande} for details). By Cor. \ref{cor:ident} and Cor. \ref{cor:identification}, there are isomorphisms of the bottom moduli spaces to the listed homogeneous spaces. In particular, the bottom map is surjective and all its fibers are isomorphic to $P_k/Q$, a rational variety. Let $f:(\P^1; 0,1,\infty) \to G/P$ be a general element in $\Mb_{0,3}(G/P, \ek)$, and $F= \Pi^{-1} (f)$. The forgetful map induces a morphism $for_F: F= \Pi^{-1}(f) \to \Pi'^{-1}(f')\simeq P_k/Q$ where $f' = for(f) \in \Mb_{0,1}(G/P,\ek)$. We claim that $for_F$ is an isomorphism, and this will finish the proof of the Theorem.

By Zariski's Main Theorem \cite[p.~209]{mumford:redbook} it suffices to show that $for_F$ is bijective. Let $\tilde{f'} \in \Pi'^{-1} (f')$ and $\tilde{f} \in for_F^{-1}(\tilde{f'}) \subset F$. Since $f$ is general we can assume that the points $f(pt_i) \in G/P$ are all distinct. By definition of stability of maps, this implies that $\tilde{f}$ is defined on $\P^1$ (rather than a tree of $\P^1$'s). Degree reasons imply that both maps $\tilde{f}: \P^1 \to Image(\tilde{f}) =  \tilde{f'}(\P^1)$ and $\pi: Image(\tilde{f}) \subset G/Q \to Image(f) \subset G/P$ are isomorphisms. By definition we must have that $\pi(\tilde{f}(pt_i)) = f(pt_i)$, and this determines $\tilde{f}$ uniquely and proves our claim.\end{proof}

\begin{remark} The previous theorem is no longer true if $(P, \alpha_k) \notin \mathcal{P}$. For example, take $G$ to be of type $B_2$ and $\Delta = \{\alpha_1,\alpha_2\}$ so that $\alpha_1$ is long. Consider the parabolic group $P$ given by $ \Delta_P = \{ \alpha_1 \}$. Notice that $G/P$ is the variety $\OG(2,5)$ of dimension $2$ isotropic planes in $\C^5$ (i.e.~lines on a smooth quadric in $\P^4$); in fact $G/P \simeq \P^3$ using Dynkin symmetry between types $B_2$ and $C_2$. One calculates that $\dim  \Mb_{0,3}(\OG(2,5), \alpha_2^\vee) = 7$. Denoting by $\OF(1,2;5)=G/B$ the full flag manifold of type $B_2$, we have that $\dim  \Mb_{0,3}(\OF(1,2;5), \alpha_2^\vee) = 6$. This shows that the map $\Pi: \Mb_{0,3} (\OF(1,2;5), \alpha_2^\vee) \to \Mb_{0,3} (\OG(2,5), \alpha_2^\vee)$ is {\bf not }surjective. If the degree $\alpha_2^\vee \in H_2(G/B)$ is replaced by $\alpha^\vee= \alpha_1^\vee + \alpha_2^\vee$ (for $\alpha = \alpha_1 + 2\alpha_2$) then $\Pi: \Mb_{0,3} (\OF(1,2;5), \alpha^\vee) \to \Mb_{0,3} (\OG(2,5), \alpha_2^\vee)$ becomes surjective.\end{remark}

We are now ready to prove the remaining parts of Thm. \ref{T:main}:

\begin{proof}[Proof of Thm. \ref{T:main}] Part (a) was proved in Cor. \ref{cor:gwimage} and cohomological triviality in $k$-free case in Thm. \ref{T:maincohtriv}.
%\begin{cor}\label{cor:cohtriv} Let $(P,\alpha_k) \in \mathcal{P}$ and $Q \subset P$ a $k$-free parabolic subgroup. Let $u,v \in W^P$ and recall the diagram \begin{equation} \xymatrix{ GW_{\ek}(\widehat{u},v) \subset \Mb_{0,3}(G/Q,\ek) \ar[r]^{\Pi} \ar[d]^{\ev_3} &  GW_{\ek}(u,v) \subset \Mb_{0,3}(G/P,\ek)\ar[d]^{\ev_3} \\ \Gamma_{\ek}(\widehat{u},v) \ar[r]^{\pi} & \Gamma_{\ek}(u,v) }
%\end{equation} If $GW_{\ek}(u,v)$ is non-empty then \[ \Pi^{-1}(GW_{\ek}(u,v)) = GW_{\ek}(\widehat{u},v)$ all maps are cohomologically trivial. In particular $\Gamma_{\ek}(u,v)$ is a projected Richardson variety, therefore it has rational singularities.
%\end{cor}
The equality $\Pi^{-1}(GW_{\ek}(u,v)) = GW_{\ek}(\widehat{u},v)$ follows from surjectivity of $\Pi$ (Thm. \ref{T:surj}) and because $\pi^{-1}X(u) = X(\widehat{u})$ and $\pi^{-1} Y(v) = Y(v)$. Since $X(u)$ and $Y(v)$ are opposite Schubert varieties, Kleiman transversality theorem \cite{kleiman} (see e.g. \cite[\S 1]{brion:Kpos}) implies that $GW_{\ek}(u,v)$ intersects the open dense set of points $f \in \Mb_{0,3}(G/P, \ek)$ which satisfy the generality conditions from the proof of Thm. \ref{T:surj}. Then the fibre over such a general point in $GW_{\ek}(u,v)$ is isomorphic to $P_k/Q$, hence it is rational. Finally, we know that $2$-point Gromov-Witten varieties are irreducible \cite[Cor. 3.3]{buch.chaput.ea:finiteness} and have rational singularities \cite[Cor. 3.1]{buch.chaput.ea:finiteness}, and invoking again Thm. \ref{T:push} yields the cohomological triviality of the top map. In particular, all maps are now surjective. By Cor.~\ref{cor:gwimage}, $\Gamma_{\ek}(\widehat{u}, v)$ is a Richardson variety, therefore $\Gamma_{\ek}(u,v)$ has rational singularities and the bottom map is cohomologically trivial by Lemma \ref{lemma:projrich}. Since the left vertical map is also cohomologically trivial (cf.~Thm.~\ref{T:maincohtriv}), and since all varieties in the diagram are normal, a standard argument based on the Grothendieck spectral sequence \cite[p.~74]{huybrechts} shows that the right vertical map is cohomologically trivial as well. \end{proof}
%{\color{blue} find the name of the spectral sequence. Grothendieck ? $R^i f R^j g$ deg to $R^{i+j} (g \circ f)$.}
%\end{proof}

\section{K-theoretic Gromov-Witten invariants} The goal of this section is to give formulas for KGW invariants $\gw{\cO^u,\cO^v, \cO^w}{\ek}$ for $X=G/P$, provided that $(P, \alpha_k) \in \mathcal{P}$.
The main result is:

\begin{thm}\label{T:mainKGW} (a) Assume that $P$ is $k$-free and $[\cF] \in K_T(X)$. Then \[ \gw{\cO^u,\cO^v, [\cF]}{\ek} = \int_{G/P} \cO^{u_k} \cdot \cO^{v_k} \cdot [\cF] = \gw{\cO^{u_k}, \cO^{v_k}, [\cF]}{0} \/. \] In particular, $ \gw{\cO^u,\cO^v, (\cO^w)^\vee}{\ek} = c_{u_k, v_k}^w$ is the structure constant in $K_T(G/P)$.

(b) Let $(P, \alpha_k) \in \mathcal{P}$,  $Q \subset P$ two parabolic  groups containing the Borel group $B$, and $[\cF], [\mathcal{G}], [\mathcal{H}] \in K_T(G/P)$. Then
 \[ \gw{[\cF],[\mathcal{G}], [\mathcal{H}]}{\ek,G/P} = \gw{\pi^*[\cF],\pi^*[\mathcal{G}],\pi^*[\mathcal{H}]}{\ek,G/Q} \] where the KGW invariants are on $G/P$ and $G/Q$ respectively and $\pi:G/Q \to G/P$ is the projection.

%(b) Let $(P,k), (Q,k) \in \mathcal{P}$ such that $P \supset Q$. Let $u,v,w \in W^P$. Then \[ \gw{\cO^u, \cO^v, \cO^w}{\ek,G/P} = \gw{\cO^u, \cO^v, \cO^w}{\ek,G/Q} \]  where the KGW invariants are on $G/P$ and $G/Q$ respectively. More generally, if $\pi:G/Q \to G/P$ is the projection, and $[\cF] \in K_T(G/P)$, then \[ \gw{\cO^u, \cO^v, [\cF]}{\ek,G/P} = \gw{\cO^u, \cO^v, \pi^*[\cF]}{\ek,G/Q} \/. \]

(c)  (K-theoretic Peterson comparison formula) Let $(P, k) \in \mathcal{P}$ such that $B  \subset P$. Then we have equalities
\[ \gw{\cO^u, \cO^v, \cO^w}{\ek, G/P} = \gw{\cO^u, \cO^v, \cO^w}{\ek, G/B} = \gw{\cO^u, \cO^v, \cO^{ww_{P_k}}}{\ek, G/B} \] where the KGW invariants are on $G/P$ and $G/B$ respectively. Here $w_{P_k}$ is the longest  element in the Weyl group $W_{P_k}$. \end{thm}

%\[ \gw{\cO^u, \cO^v, \cO_w}{\ek, G/P} = \gw{\cO^u, \cO^v, \cO_{ww'}}{\ek, G/P_k} = \gw{\cO^u, \cO^v, \cO_{ww_P w_{P_k}}}{\ek, G/B}  \] where the KGW invariants are on $G/P$, $G/P_k$ and $G/B$. Here $w'$ is the longest minimal length representative for cosets $W_{P_k}/W_P$, and $w_P,w_{P_k}$ are the longest elements in the Weyl groups $W_P$ and $W_{P_k}$. \end{thm}

Note that if $(P,\alpha_k) \in \mathcal{P}$ then $P_k$ is $k$-free, so the Theorem gives an explicit ``quantum=classical" formula for all of these KGW invariants on $G/P$. The proof of (a) and (b) will repeatedly use Thm. \ref{T:main} above;  part (c) follows from these two.

We recall next the definition and relevant properties of the divided difference operator in K-theory. This will be crucially used to obtain formulas for KGW invariants $\gw{\cO^u,\cO^v, [\cF]}{\ek}$, knowing similar formulas for $\gw{\cO_u,\cO^v, [\cF]}{\ek}$, where the Schubert varieties are opposite to each other. Note that $\cO^u = \cO_{w_0 uW_P}$ as classes in {\em non-equivariant} K-theory, but this is no longer true equivariantly. In fact, the
first type of KGW invariants appear in the definition of structure constants in the equivariant quantum K-theory.

%\subsection{The divided difference operator}\label{ss:divdiff}

Let $\alpha_k \in \Delta \setminus \Delta_P$ be a simple root and $P$ a $k$-free parabolic subgroup. Define the parabolic group $P(k)$ as in \S \ref{ss:lines}. Let $\pi_k: G/P \to G/P(k)$ be the natural projection. Its fibre is $P(k)/P \simeq \P^1$. Form the fibre diagram
\[ \xymatrix{ \mathcal{Z} \simeq G/P \times_{G/P(k)} G/P \ar[rrr]^{pr_1} \ar[d]^{pr_2 } &  & & G/P \ar[d]^{\pi_k} \\ G/P \ar[rrr]^{\pi_k}&& & G/P(k) }  \/\] where $pr_i$ are the natural projections. The {\em divided difference operator} $\partial_k: K_T(G/P) \to K_T(G/P)$ is defined by \[ \partial_k = (pr_2)_* pr_1^* = \pi_k^* (\pi_k)_* \quad \/. \] In the case when $P=B$ is the Borel subgroup, this is a famous endomorphism of $\Lambda$-algebras $K_T(X)$ satisfying $\partial_k^2 = \partial_k$ and the ``braid relations" - see e.g. \cite[Prop. 2.4]{kostant.kumar:KT}. Abusing notation we also denote by the same symbol $\partial_k$ the map sending a variety $V \subset G/P$ to $\partial_k (V) = \pi_k^{-1} (\pi_k (V)) \subset G/P$; it will be clear from the context which object is referred to.

\begin{lemma}\label{lemma:div} Let $P$ be a $k$-free parabolic subgroup. Then \[ \partial_k (\cO^w) = \cO^{w_k}; \quad \partial_k (\cO_w) = \cO_{w^k} \/. \] Moreover, if $R \subset G/P$ is a Richardson variety then $\partial_k(R)$ is irreducible and it has rational singularities, and furthermore $\partial_k[\cO_R]  = [\cO_{\partial_k(R)}] \in K_T(G/P)$. \end{lemma}

\begin{proof} The first two formulas follow from the push-forward and pull-back formulas stated in \S \ref{ss:schubert}. Further, $\pi_k(R)$ has rational singularities by Lemma \ref{lemma:projrich}. Since $\pi_k$ is a smooth morphism, it follows that $\pi_k^{-1} (\pi_k(R))$ has rational singularities as well. The formula for $\partial_k[\cO_R]$ follows again from Lemma \ref{lemma:projrich}. \end{proof}

We are now ready to prove the first two parts of Thm. \ref{T:mainKGW}:

\begin{proof}[Proof of Thm. \ref{T:mainKGW} parts (a) and (b)] The proof of part (a) is divided into two parts. We first prove that \begin{equation}\label{E:kfreeopp} \gw{\cO_u, \cO^v, [\cF]}{\ek} = \int_X \cO_{u^k} \cdot \cO^{v_k} \cdot [\cF] \/. \end{equation} Using projection formula, the calculation in (\ref{E:2GW}) above, and that $\ev_3: GW_{\ek}(u,v) \to \Gamma_{\ek}(u,v)$ is cohomologically trivial (Thm. \ref{T:maincohtriv}) we get:

 \[\begin{split} \gw{\cO_u, \cO^v, [\cF]}{\ek} = \int_{\Mb_{0,3}(X,\ek)} [\cO_{GW_{\ek}(u,v)}] \cdot \ev_3^*[\cF]  = \int_X (\ev_3)_*[\cO_{GW_d(u,v)}] \cdot [\cF]  &\\ =  \int_X \cO_{u^k} \cdot \cO^{v_k} \cdot [\cF] \/; \end{split}\] the last equality follows because $[\cO_{X(u^k) \cap Y(v_k)}] = [\cO_{u^k}] \cdot [\cO^{v_k}] \in K_T(X)$. In the non-equivariant K-theory this implies the formula in the Theorem, because $\cO_u = \cO^{w_0 uW_P}$. We now prove the equivariant  version. Since   $\{ \cO_u \}$ and $\{\cO^u\}$ are both bases for $K_T(X)$, there is an expansion $\cO^u = \sum_{z \in W^P} f_{u,z} \cO_z $, with $f_{u,z} \in \Lambda$. Applying the divided difference operator $\partial_k$ , which is a $\Lambda$-module endomorphism, and Lemma \ref{lemma:div} we obtain \[ \cO^{u_k} = \partial_k(\cO^u) = \sum_{z \in W^P} f_{u,z} \cO_{z^k} \/. \] We now use the previous identity, the fact that KGW invariant is linear in each argument and equation (\ref{E:kfreeopp}) to obtain \[ \begin{split} \gw{\cO^u, \cO^v , [\cF]}{\ek} = \sum_{z} f_{u,z} \gw{\cO_z, \cO^v, [\cF]}{\ek} =  \sum_{z} f_{u,z} \int_X \cO_{z^k} \cdot \cO^{v_k} \cdot [\cF] & \\ = \int_X (\sum_z f_{u,z} \cO_{z^k}) \cdot \cO^{v_k} \cdot [\cF] = \int_X \cO^{u_k} \cdot \cO^{v_k} \cdot [\cF] \/. \end{split} \]

We turn to the proof of part (b). Note first that if $(P,\alpha_k) \in \mathcal{P}$, $u \in W^P$ and $Q \subset P$ then clearly $(Q,\alpha_k) \in \mathcal{P}$ and $u \in W^Q$.  We have a commutative diagram
\[ \xymatrix{\Mb_{0,3}(G/Q, \ek) \ar[r]^{\Pi} \ar[d]^{\EV} & \Mb_{0,3}(G/P, \ek) \ar[d]^{\EV} \\ (G/Q)^3 \ar[r]^{\pi \times \pi \times \pi} & (G/P)^3 } \] where the top map is surjective and cohomologically trivial (cf. Thm. \ref{T:main}) and $\EV=\ev_1 \times \ev_2 \times \ev_3$. Using projection formula we get \[ \begin{split} \gw{[\cF], [\cG], [\calH]}{\ek,G/P} = \int_{\Mb_{0,3}(G/P, \ek)}\EV^*([\cF] \times [\cG] \times [\calH])  & \\ = \int_{\Mb_{0,3}(G/Q,\ek)}\Pi^*\EV^*([\cF] \times [\cG] \times [\calH]) =
\int_{\Mb_{0,3}(G/Q,\ek)}\EV^*(\pi^*[\cF] \times \pi^*[\cG] \times \pi^*[\calH]) & \\ =  \gw{\pi^*[\cF], \pi^*[\cG], \pi^*[\calH]}{\ek,G/Q} \quad \/.\end{split}  \] \end{proof}

% \end{proof}

%{\color{blue} rewrite next cor.}
 
Denote by $TX$ the tangent bundle of $X=G/P$. Then the term of degree $\ell(u) + \ell(v) - \ell(w) - \int_{\ek} c_1(TX)$ in the Taylor expansion of the KGW invariant $\gw{\cO^u, \cO^v, \cO_w}{\ek} \in R(T)$ equals the (cohomological) {\em equivariant Gromov-Witten invariant} $\gw{[Y(u)],[Y(v)],[X(w)]}{\ek}$. In turn, this equals the structure constant $c_{u,v}^{w, \ek}$ for the equivariant quantum cohomology of $G/P$ - see \cite{mihalcea:eqqcoh} for details. The next Corollary gives a formula for these invariants:

\begin{cor}\label{cor:qcgb} Let $P$ be $k$-free. Then the equivariant GW invariants satisfy:
%following identities hold: \[ \gw{\cO_u,\cO^v,\cO_w}{\varepsilon_k} = \euler{G/B}( \cO_{u^k} \otimes \cO^{v_k} \otimes \cO_w )\/ \]
\[ \gw{[Y(u)],[Y(v)],[X(w)]}{\varepsilon_k} = \left \{
\begin{array}{ll}
\int_{G/P} [Y(u_k)] \cdot [Y(v_k)] \cdot [X(w)]& \textrm{if } u_k \neq u \textrm{ and } v_k \neq v \\
0 & \textrm { otherwise}
\end{array} \right.  \]
\end{cor}
\begin{proof} The hypothesis implies that $\int_{\ek} c_1(TX) =2$.
%As in the proof of Thm. \ref{T:mainKGW}(a) we obtain that \[  \gw{[Y(u)],[Y(v)],[X(w)]}{\varepsilon_k} = \int_{G/P} [Y(u)] \cdot (\ev_1)_*[GW_{\ek}(w,v)] \/. \] But $\Gamma_{\ek}(w,v)= X(w^k) \cap Y(v_k)$. A dimension computation shows that \[ \dim GW_{\ek}(w,v) - \dim \Gamma_{\ek}(w,v) = 2- (\ell(w^k) - \ell(w) + \ell(v) - \ell(v_k)) \/. \]  If  $u_k \neq u, v_k \neq v$ it follows that $\ev_1: GW_{\ek}(w,v) \to \Gamma_{\ek}(w,v)$ is generically finite. Since $\ev_1$ is cohomologically trivial (Thm. \ref{T:maincohtriv}) it must be birational, therefore $(\ev_1)_*[GW_{\ek}(w,v)]= [\Gamma_{\ek}(w,v)]$. The first formula follows. If either $v_k \new
%
%
% and a dimension computation shows that $\dim GW_{\ek}(w,v) > \dim \ev_1(GW_{\ek}(w,v))$; thus the integral on the right is $0$.\
The equivariant GW invariant on the left vanishes unless $\ell(u) + \ell(v) \ge \ell(w) + 2$. If $u_k \neq u, v_k \neq v$ then $\ell(u_k) + \ell(v_k) \ge \ell(w)$, so the invariant on the right is the structure constant $c_{u_k,v_k}^w$ in the equivariant (non-quantum) cohomology of $G/P$. The claimed equality follows from Thm. \ref{T:mainKGW} (a). Consider the other situation; we can assume that $v_k = v$. As in the proof of Thm. \ref{T:mainKGW}(a) we obtain that \[  \gw{[Y(u)],[Y(v)],[X(w)]}{\varepsilon_k} = \int_{G/P} [Y(u)] \cdot (\ev_1)_*[GW_{\ek}(w,v)] \/. \] But $\ev_1(GW_{\ek}(w,v))= X(w^k) \cap Y(v_k)$, and a dimension computation shows that $\dim GW_{\ek}(w,v) > \dim \ev_1(GW_{\ek}(w,v))$; thus the integral on the right is $0$.\end{proof}

The next formula relates those KGW invariants on $G/P$ needed to calculate structure constants in quantum K-theory of $G/P$ (see \S \ref{s:QK} below), to structure constants in $K_T(G/Q)$, where $Q$ is $k$-free.

\begin{cor}\label{cor:xi} Let $(P, \alpha_k) \in \mathcal{P}$, and $Q$ a $k$-free parabolic with $Q \subset P$. Then for any $u,v,w \in W^P$: \[ \begin{split} \gw{\cO^u, \cO^v, \xi_w}{\ek, G/P} = & \sum_{z} \gw{\cO^u, \cO^v, \xi_z}{\ek, G/Q} =  \sum_{z} c_{u_k,v_k}^z \/. \end{split} \] where the last two sums are over minimal length representatives $z \in W^Q$ so that $z W_{P} = wW_{P}$, and $c_{u_k,v_k}^z$ are structure constants in $K_T(G/Q)$ - cf. \S \ref{ss:schubert} above. \end{cor}

%{\color{blue} sum $z$ over $W^{P_k}$.}

\begin{proof} By Thm. \ref{T:mainKGW} it follows that \[ \gw{\cO^u, \cO^v, \xi_w}{\ek, G/P} =  \gw{\cO^u, \cO^v, \pi^* (\xi_w)}{\ek, G/Q} \/. \]
To finish the proof it suffices to compute $\pi^*(\xi_w)= \sum_{z \in W^{Q}} g_{w,z} \xi_z \/. $ Then \[ g_{w,z} = \int_{G/Q} \pi^*(\xi_w) \cdot \cO^z = \int_{G/P} \xi_w \cdot \pi_*(\cO^z) = \delta_{w, z W_P} \/, \] where $\delta$ is the Kronecker delta symbol. This shows that \begin{equation}\label{E:pbxi} \pi^*(\xi_w) = \sum_{z \in W^{Q}; z W_P = w W_P} \xi_z \/. \end{equation}
and finishes the proof. \end{proof}

\subsection{Peterson comparison formula}

We prove next part (c) of Thm. \ref{T:mainKGW}. The first equality follows immediately from part (b), so we focus on the equality of first and third terms. For cohomological Gromov-Witten invariants this was conjectured by Peterson, and proved by Woodward \cite{woodward:peterson} in the non-equivariant case, and by Lam and Shimozono \cite{lam.shimozono:qaffine} equivariantly.

\begin{thm}\label{thm:PCF} Let $(P, \alpha_k) \in \mathcal{P}$ and $u,v,w \in W^P$. Then the following K-theoretic analogue of Peterson comparison formula holds:
\[ \gw{\cO^u,\cO^v,\cO^w}{\ek, G/P} = \gw{\cO^u,\cO^v, \cO^{ww_{P_k}}}{\ek,G/B} \quad \/. \] \end{thm}

\begin{proof} Denote by $\pi_P:G/B \to G/P$ and $\pi_{P_k}: G/B \to G/P_k$ the natural projections. Using that $\pi_P^*\cO^u = \cO^u$ and Thm. \ref{T:mainKGW} parts (a) and (b) we obtain \[ \gw{\cO^u,\cO^v,\cO^w}{\ek, G/P} = \int_{G/B} \cO^{{u}_k} \cdot  \cO^{{v}_k} \cdot \cO^{{w}}\quad \/. \]
By projection formula and because $u_k,v_k \in W^{P_k}$ (by Lemma \ref{lemma:minrep}, since $P_k$ is $k$-free) we have that \[ \begin{split} \int_{G/B} \cO^{{u}_k} \cdot  \cO^{{v}_k} \cdot \cO^{w} & = \int_{G/B} (\pi_{P_k})^*(\cO^{{u}_k} \cdot  \cO^{{v}_k} ) \cdot \cO^{w}  = \int_{G/{P_k}} \cO^{u_k} \cdot \cO^{v_k} \cdot (\pi_{P_k})_* \cO^{w} \\ & = \int_{G/P_k} \cO^{u_k} \cdot \cO^{v_k} \cdot (\pi_{P_k})_* \cO^{ww_{P_k}} \/, \end{split} \]  where the last equality follows from $(\pi_{P_k})_*\cO^w = \cO^{wW_{P_k}}$. Reversing the reasoning yields
\[  \int_{G/P_k} \cO^{u_k} \cdot \cO^{v_k} \cdot (\pi_{P_k})_* \cO^{ww_{P_k}} =  \int_{G/B} \cO^{{u}_k} \cdot  \cO^{{v}_k}  \cdot \cO^{ww_{P_k} }\]
and the last integer equals $\gw{\cO^u,\cO^v, \cO^{ww_{P_k}}}{\ek,G/B}$ again by Thm. \ref{T:mainKGW} (a).
\end{proof}

\begin{remark} Same line of proof can be used to show the following identities: \[ \begin{split} \gw{\cO^u,\cO^v,\cO_w}{\ek, G/P} =\int_{G/B} \cO^{{u}_k} \cdot  \cO^{{v}_k} \cdot \cO_{{w}w_P w_{P_k}} = \gw{\cO^u,\cO^v, \cO_{w w_P w_{P_k}}}{\ek, G/B} \/, \end{split} \] where $w_P$ is the longest element in $W_P$. In this identity and the one from Thm. \ref{thm:PCF} one can replace the KGW invariants with the cohomological ones and obtain identities for the (equivariant) Gromov-Witten invariants. This is because the required (in)equalities on codimensions of Schubert classes are satisfied for both invariants - see \cite{woodward:peterson} for details. Finally, Thm. \ref{T:mainKGW} part (b) implies also that $\gw{\cO^u,\cO^v,\cO_w}{\ek, G/P} = \gw{\cO^u,\cO^v,\cO_{ww_P}}{\ek, G/B}$.\end{remark}

\section{Applications to quantum K-theory}\label{s:QK} Let $X=G/P$ with $(P,\alpha_k) \in \mathcal{P}$. The formulas for the KGW invariants from Thm. \ref{T:mainKGW} allow us to compute the structure constants for the equivariant quantum K-theory of $X$ as combinations of structure constants of $K_T(G/P_k)$. In the case when $P$ is $k$-free, the projected (boundary) Gromov-Witten varieties
have rational singularities. We exploit this and certain equalities among them to prove that the structure constants of quantum K-theory either vanish or are alternating. To state precisely these results, we first recall the basic definitions of equivariant quantum K-theory, following \cite{buch.m:qk} (see also \cite{buch.chaput.ea:finiteness,lee:qk}).

Equivariant quantum K-theory of $X$, denoted by $\QK_T(X)$, is a $\Lambda[[q_i]]$-algebra, with a $\Lambda[[q_i]]$-basis $\cO^u$ where $u$ varies in $W^P$ and the parameters $q_i$ are indexed by the simple roots in $\Delta \setminus \Delta_P$. If $P$ is $k$-free then $\deg q_k= 2$; for general degrees we refer to \cite{fulton.woodward:on}. The multiplication in $\QK_T(X)$ is determined by the structure constants $N_{u,v}^{w,d} \in \Lambda$ in the identity \[ \cO^u \circ \cO^v = \sum_{w, d} N_{u,v}^{w,d}q^d \cO^w \/;\] here $d= \sum_{\alpha_i \in \Delta \setminus \Delta_P} d_i \alpha_i^\vee$ is a degree in $H_2(X)$ and $q^d = \prod q_i^{d_i}$. We are interested in the case when $d= \ek$. Then
by definition
\begin{equation}\label{E:defQK} N_{u,v}^{w,\ek} = \gw{\cO^u,\cO^v, (\cO^w)^\vee}{\ek} - \sum_{z} \gw{\cO^u,\cO^v, (\cO^{z})^\vee}{0} \cdot \gw{\cO^z, (\cO^w)^\vee}{\ek} \/. \end{equation} Here $\gw{[\cF],[\cG]}{\ek}$ denotes the $2$-point KGW invariant defined by \[ \gw{[\cF],[\cG]}{\ek} = \int_{\Mb_{0,2}(X, \ek)} \ev_1^*[\cF] \cdot \ev_2^*[\cG] \/. \] This is in fact equal to $\gw{[\cF],[\cG],\cO^{id}}{\ek}$, because the map $\Mb_{0,3}(X, \ek) \to \Mb_{0,2}(X,\ek)$ has rational fibres \cite{givental:onwdvv}. The structure constants for the non-equivariant quantum K-theory ring $\QK(X)$ are obtained by specializing $e^\lambda \to 1$ in the equivariant coefficient ring $\Lambda$; thus in this case $N_{u,v}^{w, \ek} \in \Z$. The main result of this section is the following:

\begin{thm}\label{T:QK} Let $X=G/P$ for $P$ a $k$-free parabolic group and $u,v,w \in W^P$.

(a) The structure constant $N_{u,v}^{w, \ek}$ in $\QK_T(X)$ equals:
\begin{equation}\label{E:KGW=LR} N_{u,v}^{w, \ek} = c_{u_k,v_k}^w - \delta_{w^k,ws_k} (c_{u,v}^{ws_k} + c_{u,v}^w) \/, \end{equation} where on the right hand side are structure constants in $K_T(X)$ and $\delta$ is the Kronecker delta symbol. Equivalently, $N_{u,v}^{w, \ek}$ is the coefficient of $\cO^w$ in the expansion \begin{equation}\label{E:KGW=divdiff} \partial_k(\cO^u) \cdot \partial_k(\cO^v) - \partial_k(\cO^u \cdot \cO^v) \/. \end{equation}

(b) If $u_k = u $ or $v_k = v$, the equivariant structure constant $N_{u,v}^{w, \ek} = 0$ for all $w \in W^P$.

(c) The non-equivariant structure constants $N_{u,v}^{w, \ek}$ are alternating, i.e. \begin{equation}\label{E:KGWpos} (-1)^{\ell(u) + \ell(v) - \ell(w) - \deg q_k} N_{u,v}^{w, \ek} \ge 0 \/. \end{equation}
\end{thm}

The statements (a) and (b) will be generalized to parabolic groups $P$ such that $(P, \alpha_k) \in \mathcal{P}$; see \S \ref{ss:generalP} below.
%Note that the positivity property (c) generalizes Brion's positivity \cite{brion:Kpos} for $K(G/P)$. If $w^k \neq ws_k$ then $N_{u,v}^{w, \ek} = c_{u_k, v_k}^w$ as {\em equivariant} structure constants.  The latter satisfies a positivity property conjectured by Graham-Kumar \cite{graham.kumar:positivity} and proved by Anderson-Griffeth-Miller \cite{anderson.griffeth.miller:positivity}. In the  \[ (-1)^{\ell(u) + \ell(v) - \ell(w) - \deg q_k} N_{u,v}^{w, \ek} \in \Z_{\ge 0}[e^{-\alpha_i} -1]_{\alpha_i \in \Delta \setminus \Delta_P} \/, \] which generalizes the Peterson-Graham positivity \cite{graham:pos} in equivariant cohomology, and the one in equivariant quantum cohomology proved by the second author in \cite{mihalcea:positivity}.
%A theorem of Brion \cite{brion:Kpos} , in most cases the signs of the structure constants in equation (\ref{E:KGW=LR}) do not determine that of $N_{u,v}^{w, \ek}$. Our proof is based on the geometry of the evaluation maps from Gromov-Witten varieties, and it shows why the terms of the wrong sign get canceled when we consider the full coefficient $N_{u,v}^{w, \ek}$.
We prove next the part (a) of the Theorem, and dedicate two sections to parts (b) and (c), which have more involved proofs.

\begin{proof}[Proof of Thm. \ref{T:QK} (a)] We note first that \[ \begin{split} \gw{\cO^z, (\cO^w)^\vee}{\ek} =  \gw{\cO^{id},\cO^z, (\cO^w)^\vee}{\ek} = & \int_{G/P} \cO^{id} \cdot \cO^{z_k} \cdot (\cO^w)^\vee = \delta_{z_k, w} \/, \end{split} \] where the second equality follows from Thm. \ref{T:mainKGW}. Applying the same theorem again to the remaining KGW invariants from equation (\ref{E:defQK}) yields formula (\ref{E:KGW=LR}). The second part of (a) is an easy calculation, obtained by identifying the coefficient of $\cO^w$ in the expansion (\ref{E:KGW=divdiff}) with the right hand side of (\ref{E:KGW=LR}), using that $\cO^u \cdot \cO^v = \sum c_{u,v}^w \cO^w$ in $K_T(X)$ and the formulas from Lemma \ref{lemma:div}. \end{proof}

%{\color{blue} Add a comment about the formulas and properties which generalize to more general $G/P$/}

%{\color{blue} Add similar formulas for $N_{u,v}^{w,\ek}$ in the general case $(P, \alpha_k) \in \mathcal{P}$. These are based on Cor. \ref{cor:xi}.}

\subsection{Boundary Gromov-Witten varieties and their projections}\label{ss:bdryGW} In order to prove parts (b) and (c) of Thm. \ref{T:QK}, we need to study more geometric properties of the varieties involved in the definition of the structure constants in $\QK_T(X)$. Recall from \cite[Rmk. 5.3]{buch.m:qk} that an alternate way to compute $N_{u,v}^{w,\ek}$ is:

\begin{equation}\label{E:def2} N_{u,v}^{w, \ek} = \euler{\Mb_{0,3}(X, \ek)} (\ev_1^* \cO^u  \cdot \ev_2^* \cO^v \cdot \ev_3^* (\cO^w)^\vee )- \euler{\mathcal{D}} (\ev_1^* \cO^u  \cdot \ev_2^* \cO^v \cdot \ev_3^* (\cO^w)^\vee ) \/, \end{equation}
where $\mathcal{D}$ is the fibre product \begin{equation}\label{E:bdry} \mathcal{D}= \Mb_{0,\{1,2,\bullet\}}(X, 0) \times_{X} \Mb_{0, \{ \bullet, 3 \}}(X, \ek) \simeq \Mb_{0, \{ \bullet, 3 \}}(X, \ek) \/, \end{equation} and the map to $X$ is given by evaluating at the marking $\bullet$.  (The last isomorphism holds because $\Mb_{0,3}(X,0) \simeq X$.) Set \[ GW_{0, \ek}(z,v) := GW_{\ek}(z,v)_{|\mathcal{D}}; \quad  \Gamma_{0, \ek}(z,v):=\ev_3(GW_{0,\ek}(z,v)) \/. \]  The first variety is the restriction to $\mathcal{D}$ of $GW_{\ek}(z,v)$. We refer to these as the {\em boundary GW variety} respectively {\em projected boundary GW variety}. Recall the notation $R_z^v:= X(z) \cap Y(v)$. From the identification $\mathcal{D} \simeq \Mb_{0,\{ \bullet , 3 \}}(X, \ek)$ we obtain \[ GW_{0, \ek}(z,v) = \ev_\bullet^{-1} (R_z^v); \quad  \Gamma_{0, \ek}(z,v) = \ev_3 ( \ev_\bullet^{-1} (R_z^v)) \/. \] Geometrically, $\Gamma_{0, \ek}(z,v)$ is the locus of points $x \in X$ so that there exists a line $ \ell \ni x$ of degree $\ek$ which intersects $R_z^v$. Recall that if $P$ is $k$-free then $P(k)$ denotes the parabolic group satisfying $\Delta_{P(k)} = \Delta_P \cup \{ \alpha_k\}$; denote by $\pi_k: G/P \to G/P(k)$ the natural projection. There is a remarkable coincidence between the divided difference operator $\partial_k = \pi_k^* (\pi_k)_*$ and the operator $(\ev_3)_* \ev_\bullet^*$ in the equivariant K-theory of $G/P$, which we explain next.

\begin{prop}\label{prop:coincidence} Let $P$ be a $k$-free parabolic group and $v,z \in W^P$.

(a) There is a natural isomorphism $\mathcal{D} \simeq G/P \times_{G/P(k)} G/P$. Under this isomorphism, the evaluation maps $\ev_\bullet$ and $\ev_3$ correspond to projections $pr_1,pr_2$ in the fibre diagram

\[ \xymatrix{ \mathcal{D} \simeq G/P \times_{G/P(k)} G/P \ar[rrr]^{\ev_\bullet = pr_1} \ar[d]^{\ev_3 = pr_2 } &  & & G/P \ar[d]^{\pi_k} \\ G/P \ar[rrr]^{\pi_k}&& & G/P(k) }  \/. \]

(b) Assume that $R_z^v$ is non-empty. Then the Gromov-Witten subvariety $GW_{0,\ek}(z,v)$ is isomorphic to $pr_1^{-1} R_z^v$ and the diagram in (a) determines another fibre diagram

\[ \xymatrix{ GW_{0,\ek}(z,v) \simeq pr_1^{-1} R_z^v \ar[rrr]^{\ev_\bullet = pr_1} \ar[d]^{\ev_3 = pr_2 } &  & & R_z^v \ar[d]^{\pi_k} \\ \Gamma_{0,\ek}(z,v) \simeq  \pi_k^{-1}(\pi_k R_z^v) \ar[rrr]^{\pi_k}&& & \pi_k(R_z^v) }  \]
Furthermore, $\Gamma_{0,\ek}(z,v)$ has rational singularities and the map $\ev_3:GW_{0,\ek}(z,v) \to \Gamma_{0, \ek}(z,v)$ is cohomologically trivial.

(c) Under the identifications in (a), the divided difference operator $\partial_k = \pi_k^* (\pi_k)_*$ equals $(\ev_3)_* \ev_\bullet^*$ and it satisfies \[ \partial_k [\cO_{R_z^v}] = [\cO_{\partial_k(R_z^v)}] = [\cO_{\Gamma_{0, \ek}(z,v)}] \/. \]

\end{prop}

\begin{proof} We first prove (a). Let $x \in G/P$. We will use repeatedly the observation that the unique line $\ell_x \ni x$ of degree $\ek$ specified by Cor. \ref{cor:kfreelines} is given by $\ell_x = \pi_k^{-1} (\pi_k(x))$. This follows from Cor. \ref{cor:identification}. Then the evaluation  maps $\ev_\bullet, \ev_3: \mathcal{D} \to G/P$ have the property that $\pi_k \ev_\bullet = \pi_k \ev_3$. This determines a morphism \[ \Psi: \mathcal{D} \to G/P \times_{G/P(k)} G/P\]  sending a line $\ell$ with two markings $(x_\bullet, x_3)$, to $(x_\bullet, x_3)$. A line is uniquely determined by any of the two points it contains, thus this morphism is injective. $\Psi$ is also surjective, because $\dim \mathcal{D} = \dim G/P \times_{G/P(k)} G/P$, and the latter variety is irreducible. Since both varieties are normal, Zariski's Main Theorem \cite[p.~ 209]{mumford:redbook} implies that $\Psi$ is an isomorphism. The identification of the evaluation maps with projections $pr_1, pr_2$ is immediate from the definition of $\Psi$.

We now prove part (b). The restriction $\Psi'$ of $\Psi$ to $GW_{0, \ek}(z,v)$ determines maps to $R_z^v$ and $\Gamma_{0,\ek}(z,v) \subset \pi_k^{-1} (\pi_k R_z^v) (=\partial_k(R_z^v)$), and therefore a map \[ \Psi': GW_{0, \ek}(z,v) \to R_z^v \times_{\pi_k (R_z^v)} \pi_k^{-1} (\pi_k (R_z^v)) \simeq pr_1^{-1} (R_z^v) \/. \] The morphism $\Psi'$ is again injective. Since $R_z^v$ is normal, irreducible and $pr_1$ is a $\P^1$-bundle projection it follows that $pr_1^{-1} (R_z^v)$ is normal and irreducible as well. We know that $GW_{0,\ek}(z,v)$ has rational singularities \cite[Cor.~3.1]{buch.chaput.ea:finiteness} hence it is normal. To prove that $\Psi'$ is an isomorphism it suffices to show that $\dim GW_{0, \ek}(z,v) = \dim pr_1^{-1} (R_z^v)$. But a standard application of Kleiman's transversality theorem \cite{kleiman} shows that \[ \begin{split} \dim GW_{0,\ek}(z,v) = \dim \Mb_{0,\{ \bullet, 3 \}} (G/P, \ek) - \codim R_z^v =& \\  \dim G/P + 1 - (\dim G/P - \dim R_z^v) = \dim R_z^v +1 = \dim pr_1^{-1} (R_z^v) \/. \end{split}  \] Since the projection map $pr_2$ is now identified with the restriction of the evaluation map $\ev_3$ to $GW_{0,\ek}(z,v)$ it follows that $\pi_k^{-1} (\pi_k(R_z^v))$ is isomorphic to the image of $\ev_3$, which is $\Gamma_{0, \ek}(z,v)$. Lemma \ref{lemma:projrich} gives that $\partial_k(R_z^v)$ has rational singularities. We employ Thm. \ref{T:push} to prove that $\ev_3$ is cohomologically trivial. Because we work in a fibre square diagram, a general fibre of $\ev_3$ coincides with a general fibre $F$ of $\pi_k:R_z^v \to \pi_k(R_z^v)$. This morphism is cohomologically trivial (Lemma \ref{lemma:projrich}), in particular it has connected fibres. But the fibre of the unrestricted morphism $\pi_k: G/P \to G/P(k)$ is $\P^1$, and this implies that $F$ is either isomorphic to $\P^1$ or to a reduced point. The hypotheses of Thm. \ref{T:push} are then satisfied and $\ev_3$ must be also cohomologically trivial.

We now prove part (c). By \cite[III.9.3]{hartshorne} or \cite[5.3.15]{chriss.ginzburg} it follows that in the given fibre diagram $\partial_k = \pi_k^* (\pi_k)_* = (\ev_3)_* \ev_\bullet^*$ as operators in the equivariant K-theory. Then $\partial_k([\cO_{R_z^v}]) = \cO_{\partial_k(R_z^v)}$ from Lemma \ref{lemma:projrich}. \end{proof}

We prove in the next three lemmas the key facts needed in the proof of positivity and vanishing of the structure constants $N_{u,v}^{w,\ek}$. In what follows we will repeatedly use the description of $\Gamma_{0, \ek}(z,v)$ as the locus of points $ x \in G/P$ contained in a line $\ell$ of degree $\ek$ such that $\ell \cap R_z^v \neq \emptyset$. In all the lemmas we use the common hypothesis that $P$ is a $k$-free parabolic subgroup and $z,v \in W^P$.

\begin{lemma}\label{lemma:G01eq} If either $z^k=z$ or $v_k=v$ then $\Gamma_{0,\ek}(z,v) = \Gamma_{\ek}(z,v)$. \end{lemma}

\begin{proof} One inclusion is clear. Take $x \in \Gamma_{\ek}(z,v)$ and $\ell_x \ni x$ the unique line from Cor. \ref{cor:kfreelines}. By definition of $\Gamma_{\ek}(z,v)$, the line $\ell_x$ intersects both $X(z)$ and $Y(v)$. Assume now that $z^k=z$. Then by Prop. \ref{prop:kfreelines} this line is entirely included in $X(z)$, therefore it must intersect $X(z) \cap Y(v)$. The case when $v_k =v$ is similar. \end{proof}

\begin{lemma}\label{lemma:ample} Assume that both $z^k \neq z$ and $v_k \neq v$. Then both evaluation maps $\ev_3:GW_{\ek}(z,v) \to \Gamma_{\ek}(z,v)$ and $\ev_3: GW_{0,\ek}(z,v) \to \Gamma_{0,\ek}(z,v)$ are birational and $\Gamma_{0, \ek}(z,v)$ is a divisor in $\Gamma_{\ek}(z,v)$. In particular, we have strict inclusions \[ R_z^v = X(z) \cap Y(v) \subset \Gamma_{0, \ek}(z,v) \subset X(zs_k) \cap Y(vs_k) = R_{zs_k}^{vs_k} \/. \] \end{lemma}

\begin{proof} Thm. \ref{T:maincohtriv} shows that $\Gamma_{\ek}(z,v) = X(z^k) \cap Y(v_k)$. The hypothesis on $z,v$ implies that \[ \dim \Gamma_{\ek}(z,v) = \dim X(z^k) \cap Y(v_k) = \ell(z) - \ell(v) + 2  = \dim GW_{\ek}(z,v)\/;\]the last equality follows from a standard calculation based on Kleiman Transversality Theorem \cite{kleiman}. Since $\ev_3:GW_{\ek}(z,v) \to \Gamma_{\ek}(z,v)$ is cohomologically trivial by Thm. \ref{T:maincohtriv}, Stein factorization \cite[III.11.5]{hartshorne} shows that it must be birational.

We now turn to $\Gamma_{0,\ek}(z,v)$. The fibre square in the part (b) of Prop. \ref{prop:coincidence} implies that a general fibre of $\ev_3:GW_{0,\ek}(z,v) \to \Gamma_{0,\ek}(z,v)$ is isomorphic to a general fibre of $\pi_k:R_z^v \to \pi_k(R_z^v)$. It suffices to show that the latter map is birational. Indeed, the hypothesis on $z$ means that $z \in W^{P(k)}$, therefore $\pi_k:X(z) \to \pi_k(X(z)) = X(z)$ is birational; let $U \subset X(z) \subset G/P$ be the open set where this map is an isomorphism. Because $X(z) \cap Y(v) \neq \emptyset$, Kleiman Transversality Theorem implies that $Y(v) \cap U \neq \emptyset$; the claimed birationality follows.

Finally, the definition of $\Gamma_{0,\ek}(z,v)$ implies that we always have inclusions (not necessarily strict)
\[ R_z^v = X(z) \cap Y(v) \subset \Gamma_{0, \ek}(z,v) \subset \Gamma_{\ek}(z,v) = R_{zs_k}^{vs_k} \/. \]
Counting dimensions again we obtain that $\dim R_{zs_k}^{vs_k} - \dim R_z^v = 2$ and $\dim \Gamma_{\ek}(z,v) - \dim \Gamma_{0,\ek}(z,v) =1$, therefore the inclusions must be strict.
 \end{proof}

\begin{remark} The proof of the previous lemma shows more: $\ev_3:GW_{0,\ek}(z,v) \to \Gamma_{0,\ek}(z,v)$ is birational whenever $z^k \neq z$ {\em or} $v_k \neq v$. \end{remark}

\subsection{Vanishing and positivity of $N_{u,v}^{w, \ek}$} To prove the vanishing and positivity parts from Thm. \ref{T:QK} we need to further rewrite the definition of the structure constants from equation (\ref{E:def2}), using the additional information from the previous section. Note that a theorem of Brion \cite{brion:Kpos} gives the signs of the structure constants in equation (\ref{E:KGW=LR}); but these {\em do not} determine that of $N_{u,v}^{w, \ek}$. Our proof uses the geometric results from the previous section, and it shows why the terms of the wrong sign get canceled when we consider the full coefficient $N_{u,v}^{w, \ek}$. In this section $X=G/P$ where $P$ is a $k$-free parabolic subgroup. Expand: \[ \cO^u = \sum_z P_{u,z} \cO_z; \qquad P_{u,z} \in \Lambda \/. \]  We replace $\cO^u$ by $\sum_z P_{u,z} \cO_z$ in (\ref{E:def2}) to obtain
\begin{equation}\label{E:def3} \begin{split} N_{u,v}^{w, \ek} = \sum_{z} P_{u,z} \bigl(\euler{\Mb_{0,3}(X, \ek)} (\ev_1^* \cO_z  \cdot \ev_2^* \cO^v \cdot \ev_3^* (\cO^w)^\vee )- & \\\euler{\mathcal{D}} (\ev_1^* \cO_z  \cdot \ev_2^* \cO^v \cdot \ev_3^* (\cO^w)^\vee )\bigr) \/. \end{split} \end{equation} We noticed in \S \ref{ss:KGW} that $\ev_1^*\cO_z \cdot \ev_2^* \cO^v = [\cO_{GW_{\ek}(z,v)}]$ in $K_T(\Mb_{0,3}(X, \ek))$. Same proof shows that in $K_T(\mathcal{D})$ we have $\ev_1^*\cO_z \cdot \ev_2^*\cO^v =  [\cO_{GW_{0,\ek}(z,v)}]$. Recall that both morphisms $\ev_3:GW_{\ek}(z,v) \to \Gamma_{\ek}(z,v)$ and $\ev_3:GW_{0,\ek}(z,v) \to \Gamma_{0,\ek}(z,v)$ are cohomologically trivial (Thm. \ref{T:maincohtriv} respectively Prop. \ref{prop:coincidence}). Then by projection formula we have

\[ \begin{split}  \euler{\Mb_{0,3}(X, \ek)} (\ev_1^* \cO_z  \cdot \ev_2^* \cO^v \cdot \ev_3^* (\cO^w)^\vee )- \euler{\mathcal{D}} (\ev_1^* \cO_z  \cdot \ev_2^* \cO^v \cdot \ev_3^* (\cO^w)^\vee ) &\\ = \int_{X} ( [\cO_{\Gamma_{\ek}(z,v)}] -  [\cO_{\Gamma_{0,\ek}(z,v)}]) \cdot (\cO^w)^\vee \end{split} \] therefore \begin{equation}\label{E:def4} N_{u,v}^{w,\ek} = \sum_z P_{u,z} \int_{X} ( [\cO_{\Gamma_{\ek}(z,v)}] -  [\cO_{\Gamma_{0,\ek}(z,v)}]) \cdot (\cO^w)^\vee \/. \end{equation}

This immediately implies the claimed vanishing:

\begin{proof}[Proof of Thm. \ref{T:QK} part (b)] Because $\QK_T(X)$ is a commutative ring, we can assume that ${v}_k = {v}$. The parabolic subgroup $P$ is $k$-free, thus $\Gamma_{\ek}(z,v) = \Gamma_{0,\ek}(z,v)$ by Lemma \ref{lemma:G01eq}. Then the vanishing follows from the identity (\ref{E:def4}).\end{proof}

For the proof of Thm. \ref{T:QK}(c) we will  use the following positivity theorem of Brion (\cite[Thm.1]{brion:Kpos}).
% which is needed to show the positivity of $N_{u,v}^{w, \ek}$:

\begin{thm}(Brion)\label{T:brion} Let $Y$ be an irreducible closed subvariety of $X=G/P$ which has rational singularities. Consider the expansion \[ [\cO_Y] = \sum a_w \cO_w \/. \] Then $(-1)^{\dim Y - \ell(w)} a_w \ge 0$. \end{thm}

\begin{proof}[Proof of Thm. \ref{T:QK} part (c)] Given the vanishing in Thm. \ref{T:QK}(b), we can assume that $u_k \neq u$ and $v_k \neq v$; equivalently, $ \ell(u s_k) = \ell(u) -1 $ and $\ell(vs_k) = \ell(v) -1 $.  There are two cases. The first is   when $w^k =w$. From Thm. \ref{T:QK} part (a) $N_{u,v}^{w,\ek} = c_{us_k, vs_k}^w$, and this coefficient satisfies the required positivity by Thm. \ref{T:brion}.%Brion's main result in \cite{brion:Kpos}.

We turn to the second case when $w^k \neq w$. In the non-equivariant case $\cO^u = \cO_{\wt{u}}$, where $\wt{u} \in W^P$ is the minimal length representative in the coset $w_0 u W_P \in W/W_P$. Then equation (\ref{E:def4}) becomes:

\[ N_{u,v}^{w,\ek} = \int_X [\cO_{\Gamma_{\ek}(\wt{u},v)}] \cdot (\cO^w)^\vee - \int_X [\cO_{\Gamma_{0,\ek}(\wt{u},v)}] \cdot (\cO^w)^\vee \/. \]

Then the required positivity follows again from Brion's Theorem. To see that, consider the expansions \[ [\cO_{\Gamma_{\ek}(\wt{u},v)}] = \sum_z d_{\wt{u},v}^z \cO^z; \quad [\cO_{\Gamma_{0,\ek}(\wt{u},v)}] = \sum_z f_{\wt{u},v}^z \cO^z \/. \]
Then \begin{equation}\label{E:def5} N_{u,v}^{w,\ek} = d_{\wt{u},v}^w - f_{\wt{u},v}^w \/, \end{equation} and by Brion's Theorem \[ (-1)^{\dim \Gamma_{\ek}(\wt{u},v) - \dim Y(w)} d_{\wt{u},v}^w \ge 0 \textrm{ and }  (-1)^{\dim \Gamma_{0,\ek}(\wt{u},v) - \dim Y(w)} f_{\wt{u},v}^w \ge 0 \/. \]
But the hypothesis on $u$ implies that $\wt{u}^k \neq \wt{u}$, therefore by Lemma \ref{lemma:ample} \[ \dim \Gamma_{0,\ek}(\wt{u},v) = \dim GW_{0,\ek}(\wt{u},v) =  \dim R_{\wt{u}}^v + 1 = \ell(\wt{u}) -\ell(v) + 1 \/; \] (the second equality follows from the fibre square in Prop. \ref{prop:coincidence}). Same Lemma implies that $\dim \Gamma_{\ek}(\wt{u},v) = \dim \Gamma_{0,\ek}(\wt{u},v) + 1$. Finally, \[ \begin{split} \dim  \Gamma_{\ek}(\wt{u},v) - \dim Y(w) = \ell(\wt{u}) - \ell(v) + 2 - \dim Y(w) =& \\ (\dim X - \ell(u)) - \ell(v) + 2 - (\dim X - \ell(w)) = \ell(w) + 2 - (\ell(u) + \ell(v)) \/. \end{split} \] This shows that both terms $d_{\wt{u},v}^w$ and $f_{\wt{u},v}^w$ in (\ref{E:def5}) have the correct sign, and finishes the proof. \end{proof}

\begin{remark}\label{remark:KTpos} We conjecture that the {\em equivariant} coefficients $N_{u,v}^{w, \ek}$ satisfy the positivity property: \[ (-1)^{\ell(u) + \ell(v) - \ell(w) - \deg q_k} N_{u,v}^{w, \ek} \in \Z_{\ge 0} [e^{-\alpha_i} -1]_{\alpha_i \in \Delta \setminus \Delta_P} \/. \]  This holds if $P$ is $k$-free and $w^k \neq ws_k$, because then $N_{u,v}^{w, \ek} = c_{u_k, v_k}^w$ as {\em equivariant} structure constants.  The latter satisfies the claimed positivity thanks to a conjecture of Griffeth-Ram \cite{griffeth.ram:affine} (see also \cite{graham.kumar:positivity}) proved by Anderson-Griffeth-Miller \cite{anderson.griffeth.miller:positivity}. It generalizes the Peterson-Graham positivity \cite{graham:pos} in equivariant cohomology, and the one in equivariant quantum cohomology proved by the second author in \cite{mihalcea:positivity}.

If $P$ is $k$-free, this conjecture follows if the positivity theorem \cite[Thm. 4.1]{anderson.griffeth.miller:positivity} would generalize from $\partial$ an ample divisor to a big and nef divisor. We will consider this generalization and its relation to the QK structure constants elsewhere.\end{remark}

\subsection{Structure constants in $\QK_T(G/P)$ for more general parabolic groups}\label{ss:generalP} Let $X=G/P$. In this section we state statements analogous to (a) and (b) from Thm.~\ref{T:QK} in the case when $P$ is not $k$-free, but $(P, \alpha_k) \in \mathcal{P}$. The proofs will be similar to those before, therefore we only sketch them and point out the differences. The main tools we used in the case when $P$ is $k$-free - Prop. \ref{prop:coincidence} and lemmas after that - are no longer available. Instead, we will rely on a weaker version of Thm. \ref{T:main} to transfer the computations from $G/P$ to $G/P_k$. Since $P_k$ is $k$-free, most - but not all - calculations from previous section will extend.

To fix notation denote by $\pi:G/P_k \to G/P$ the projection and by $\Pi:\Mb_{0,3}(G/P_k, \ek) \to \Mb_{0,3}(G/P, \ek)$ the map induced by $\pi$. If $u \in W^P$, recall that $\widehat{u} \in W^{P_k}$ is defined by $\pi^{-1} X(u) = X(\widehat{u})$ and that $(\cO^u)^\vee$ is denoted by $\xi_u$. We keep the notation from \S \ref{ss:bdryGW}, but we distinguish by $\mathcal{D}_P$ respectively $\mathcal{D}_{P_k}$ the boundary loci defined in (\ref{E:bdry}) for maps to $G/P$ and $G/P_k$. We have a commutative diagram \begin{equation}\label{diagram:general} \xymatrix{ \mathcal{D}_{P_k} \subset \Mb_{0,3}(G/P_k,\ek) \ar[r]^{\Pi} \ar[d]^{\EV} & \mathcal{D}_P \subset \Mb_{0,3}(G/P, \ek) \ar[d]^{\EV}\\ (G/P_k)^3 \ar[r]^{\pi \times \pi \times \pi} & (G/P)^3 } \end{equation} where as usual $\EV = \ev_1 \times \ev_2 \times \ev_3$. The top arrow means  that $\Pi$ induces a restriction map $\Pi: \mathcal{D}_{P_k} \to \mathcal{D}_P$. We know from Thm. \ref{T:surj} that $\Pi:\Mb_{0,3}(G/P_k, \ek) \to \Mb_{0,3}(G/P, \ek)$ is surjective and cohomologically trivial, and minor modifications in that proof show that its restriction $\Pi: \mathcal{D}_{P_k} \to \mathcal{D}_P$ has the same properties. Using this, projection formula, and equation (\ref{E:def2}) we obtain (for $u,v,w \in W^P$): \begin{equation}\label{E:pb} \begin{split} N_{u,v}^{w, \ek}  = \euler{\Mb_{0,3}(G/P, \ek)}(\EV^*(\cO^u \times \cO^v \times \xi_w)) - \euler{\mathcal{D}_P}(\EV^*(\cO^u \times \cO^v \times \xi_w))  = & \\ \euler{\Mb_{0,3}(G/P_k, \ek)}(\EV^*(\cO^u \times \cO^v \times \pi^*(\xi_w))) - \euler{\mathcal{D}_{P_k}}(\EV^*(\cO^u \times \cO^v \times \pi^*(\xi_w))) \/. \end{split} \end{equation}
A standard computation based on the class of the diagonal $\Delta \subset G/P_k \times G/P_k$ (see e.g. \cite[Lemma. 5.1]{buch.chaput.ea:finiteness}) shows that \[ \begin{split} \euler{\mathcal{D}_{P_k}}(\EV^*(\cO^u & \times \cO^v \times \pi^*(\xi_w))) = \sum_{z \in W^{P_k}} \gw{\cO^u,\cO^v, \xi_z}{0,G/P_k} \gw{\cO^z, \pi^*(\xi_w)}{\ek,G/P_k}  \\ =  \sum_z c_{u_k,v_k}^z \/, \end{split} \] where the last sum is over $z \in W^{P_k}$ such that  $z_k W_P=w W_P$ as cosets in $W/W_P$. This follows because $\gw{\cO^u,\cO^v, \xi_z}{0,G/P_k} = c_{u_k,v_k}^z$ and \[ \gw{\cO^z, \pi^*(\xi_w)}{\ek,G/P_k} = \int_{G/P_k} \cO^{z_k} \cdot \pi^*(\xi_w) \/; \] then we use formula (\ref{E:pbxi}) for $\pi^*(\xi_w)$ in Cor. \ref{cor:xi} above. Together with formula (\ref{E:pb}) this implies:

\begin{thm}\label{T:formulagenP} Let $(P, \alpha_k) \in \mathcal{P}$, and $u,v,w \in W^P$. The structure constants $N_{u,v}^{w, \ek}$ in $\QK_T(G/P)$ are given by the following formula:
\[ N_{u,v}^{w, \ek} = \sum_{a} c_{u_k, v_k}^a - \sum_{b} c_{u,v}^b \] where the first sum is over $a \in W^{P_k}$ such that $aW_P = w W_P$, the second over those $b \in W^{P_k}$ such that $b_k W_P = w W_P$, and $c_{u_k,v_k}^a,c_{u,v}^b$ are structure constants in $K_T(G/P_k)$.\end{thm}

We now turn to the analogue of vanishing result from (b), Thm. \ref{T:QK}.

\begin{thm}\label{T:vanishing} Let $(P,\alpha_k) \in \mathcal{P}$ and $u, v \in W^P$.  Assume that either $u_k = u $ or $v_k = v$. Then $N_{u,v}^{w, \ek} = 0$ in $\QK_T(G/P)$. \end{thm}

\begin{proof} By commutativity of $\QK_T(X)$ we can assume that $v_k = v$. Starting from identity (\ref{E:pb}), the same reasoning used to obtain formula (\ref{E:def4}) shows that \begin{equation}\label{E:defgen} N_{u,v}^{w,\ek} = \sum_{z \in W^{P_k}} P_{u,z} \int_{G/P_k} ( [\cO_{\Gamma_{\ek}(z,v)}] -  [\cO_{\Gamma_{0,\ek}(z,v)}]) \cdot \pi^*(\xi_w)\/, \end{equation} where
 $\cO^u = \sum_{z \in W^{P_k}} P_{u,z} \cO_z \in K_T(G/P_k)$. But then $\Gamma_{\ek}(z,v) = \Gamma_{0,\ek}(z,v)$ by Lemma \ref{lemma:G01eq} and we are done. \end{proof}

%\newpage
\section{Appendix} In this appendix we present multiplication tables for $QK_T(G/B)$ up to degrees $\ek$, in the case when $G=\SL_3(\C)$ and $G= Sp_4(\C)$. In type $A$, the multiplications are up to the symmetry $s_1 \longleftrightarrow s_2$ of the Weyl group; in type $C$, $\alpha_1$ denotes the short root.

%\newpage

%\subsection{Table of $\QK_T(\SL_2(\C)/B)$ up to degrees $\ek$}
%{\small
{\bf Table of $\QK_T(\SL_3(\C)/B)$ up to degrees $\ek$}

%\begin{align*}
%   \cO^{s_1}\circ \cO^{s_1}&\equiv (1-e^{-\alpha_1})\cO^{s_1}+e^{-\alpha_1}\cO^{s_2s_1}+e^{-\alpha_1}q_1-e^{-\alpha_1}q_1\cO^{s_2}\\
%   \cO^{s_1}\circ \cO^{s_2}&\equiv\cO^{s_1s_2}+\cO^{s_2s_1}- \cO^{s_1s_2s_1}\\
%   \cO^{s_1}\circ \cO^{s_1s_2}&\equiv(1- e^{-\alpha_1})\cO^{s_1s_2}+e^{-\alpha_1}\cO^{s_1s_2s_1}\\
%     \cO^{s_1}\circ \cO^{s_2s_1}&\equiv(1-e^{-\alpha_1-\alpha_2})\cO^{s_2s_1}+e^{-\alpha_1-\alpha_2}q_1\cO^{s_2}\\
%     \cO^{s_1}\circ \cO^{s_1s_2s_1}&\equiv(1-e^{-\alpha_1-\alpha_2})\cO^{s_1s_2s_1}+e^{-\alpha_1-\alpha_2}q_1\cO^{s_1s_2}\\
%      \cO^{s_1s_2}\circ \cO^{s_1s_2}&\equiv(1-e^{-\alpha_1})(1-e^{-\alpha_1-\alpha_2})\cO^{s_1s_2}+e^{-\alpha_1}q_2\cO^{s_2s_1}+(1-e^{-\alpha_1})e^{-\alpha_1-\alpha_2}q_2\cO^{s_1}\\
%   \cO^{s_1s_2}\circ \cO^{s_2s_1}&\equiv(1-e^{-\alpha_1-\alpha_2})\cO^{s_1s_2s_1}\\
%    \cO^{s_1s_2}\circ  \cO^{s_1s_2s_1}&\equiv(1-e^{-\alpha_1})(1-e^{-\alpha_1+-\alpha_2})\cO^{s_1s_2s_1}+(1-e^{-\alpha_1-\alpha_2})e^{-\alpha_1}q_2\cO^{s_2s_1}\\
% \cO^{s_1s_2s_1}\circ  \cO^{s_1s_2s_1}&\equiv(1-e^{-\alpha_1})(1-e^{-\alpha_2})(1-e^{-\alpha_1-\alpha_2})\cO^{s_1s_2s_1}\\
%                       &\quad +(1-e^{-\alpha_1})(1-e^{-\alpha_1-\alpha_2})e^{-\alpha_2}q_1\cO^{s_1s_2}
%                               + (1-e^{-\alpha_2})(1-e^{-\alpha_1-\alpha_2})q_2\cO^{s_2s_1}\\
%  \end{align*}
%}
%Note $\cO^{\scriptsize{\mbox{id}}}$ is the identity element in $QK^*_T(G/B)$.
%
%For  type $A_2$, the multiplication table,  up to degree one structure constants and up to the symmetry of Dynkin diagram, is given by
\resizebox{.99\hsize}{!}{\parbox{7cm}{\begin{align*}
   \cO^{s_1}\circ \cO^{s_1}&\equiv (1-e^{-\alpha_1})\cO^{s_1}+e^{-\alpha_1}\cO^{s_2s_1}+e^{-\alpha_1}q_1-e^{-\alpha_1}q_1\cO^{s_2}\\
   \cO^{s_1}\circ \cO^{s_2}&\equiv\cO^{s_1s_2}+\cO^{s_2s_1}-\cO^{s_1s_2s_1}\\
   \cO^{s_1}\circ \cO^{s_1s_2}&\equiv(1-e^{-\alpha_1})\cO^{s_1s_2}+e^{-\alpha_1}\cO^{s_1s_2s_1}\\
     \cO^{s_1}\circ \cO^{s_2s_1}&\equiv(1-e^{-\alpha_1-\alpha_2})\cO^{s_2s_1}+e^{-\alpha_1-\alpha_2}q_1\cO^{s_2}\\
     \cO^{s_1}\circ \cO^{s_1s_2s_1}&\equiv(1-e^{-\alpha_1-\alpha_2})\cO^{s_1s_2s_1}+e^{-\alpha_1-\alpha_2}q_1\cO^{s_1s_2}\\
      \cO^{s_1s_2}\circ \cO^{s_1s_2}&\equiv(1-e^{-\alpha_1})(1-e^{-\alpha_1-\alpha_2})\cO^{s_1s_2}+e^{-\alpha_1}q_2\cO^{s_2s_1}\\
         &\quad +(1-e^{-\alpha_1})e^{-\alpha_1-\alpha_2}q_2\cO^{s_1}\\
   \cO^{s_1s_2}\circ \cO^{s_2s_1}&\equiv(1-e^{-\alpha_1-\alpha_2})\cO^{s_1s_2s_1}\\
    \cO^{s_1s_2}\circ  \cO^{s_1s_2s_1}&\equiv(1-e^{-\alpha_1})(1-e^{-\alpha_1-\alpha_2})\cO^{s_1s_2s_1}+(1-e^{-\alpha_1-\alpha_2})e^{-\alpha_1}q_2\cO^{s_2s_1}\\
 \cO^{s_1s_2s_1}\circ  \cO^{s_1s_2s_1}&\equiv(1-e^{-\alpha_1})(1-e^{-\alpha_2})(1-e^{-\alpha_1-\alpha_2})\cO^{s_1s_2s_1}\\
                       &\quad +(1-e^{-\alpha_1})(1-e^{-\alpha_1-\alpha_2})e^{-\alpha_2}q_1\cO^{s_1s_2} \\
                         &\quad      + (1-e^{-\alpha_2})(1-e^{-\alpha_1-\alpha_2})q_2\cO^{s_2s_1}\\
  \end{align*}}}

%}}
%\subsection{Table of $\QK_T(\Sp_4(\C)/B)$ up to degrees $\ek$}
 {\bf Table of $\QK_T(\Sp_4(\C)/B)$ up to degrees $\ek$}

%\vspace{-1.6cm}
\resizebox{.99\hsize}{!}{\parbox{8cm}{\begin{align*}
   \cO^{s_1}\circ \cO^{s_1}&\equiv (1-e^{-\alpha_1})\cO^{s_1}+e^{-\alpha_1}\cO^{s_2s_1}+e^{-\alpha_1}q_1-e^{-\alpha_1}q_1\cO^{s_2}\\
   \cO^{s_1}\circ \cO^{s_2}&\equiv\cO^{s_1s_2}+\cO^{s_2s_1}-\cO^{s_1s_2s_1}-\cO^{s_2s_1s_2}+\cO^{s_1s_2s_1s_2}\\
   \cO^{s_1}\circ \cO^{s_1s_2}&\equiv(1-e^{-\alpha_1})\cO^{s_1s_2}+e^{-\alpha_1}\cO^{s_1s_2s_1}+e^{-\alpha_1}\cO^{s_2s_1s_2}-e^{-\alpha_1}\cO^{s_1s_2s_1s_2}\\
  %\end{align*}}}
     %\newpage
    %  \resizebox{.99\hsize}{!}{\parbox{8cm}{\begin{align*}
      \cO^{s_1}\circ \cO^{s_2s_1}&\equiv(1-e^{-\alpha_1-\alpha_2})\cO^{s_2s_1}+e^{-\alpha_1-\alpha_2}\cO^{s_1s_2s_1}\\
                                  &\quad +e^{-\alpha_1-\alpha_2}q_1\cO^{s_2}-e^{-\alpha_1-\alpha_2}q_1\cO^{s_1s_2}\\
     \cO^{s_1}\circ \cO^{s_1s_2s_1}&\equiv(1-e^{-2\alpha_1-\alpha_2})\cO^{s_1s_2s_1}+e^{-2\alpha_1-\alpha_2}q_1\cO^{s_1s_2}\\
     \cO^{s_1}\circ \cO^{s_2s_1s_2}&\equiv(1-e^{-\alpha_1-\alpha_2})\cO^{s_2s_1s_2}+e^{-\alpha_1-\alpha_2}\cO^{s_1s_2s_1s_2}\\
     \cO^{s_1}\circ \cO^{s_1s_2s_1s_2}&\equiv(1-e^{-2\alpha_1-\alpha_2})\cO^{s_1s_2s_1s_2}+e^{-2\alpha_1-\alpha_2}q_1\cO^{s_2s_1s_2}\\
   \cO^{s_2}\circ \cO^{s_2}&\equiv (1-e^{-\alpha_2})\cO^{s_2}+(1+e^{-\alpha_1})e^{-\alpha_2}\cO^{s_1s_2}-e^{-\alpha_1-\alpha_2}\cO^{s_2s_1s_2}\\
                       &\quad +e^{-\alpha_2}q_2-(1+e^{-\alpha_1})e^{-\alpha_2}q_2\cO^{s_1}+e^{-\alpha_1-\alpha_2}q_2\cO^{s_2s_1}\\
     \cO^{s_2}\circ \cO^{s_1s_2}&\equiv(1-e^{-2\alpha_1-\alpha_2})\cO^{s_1s_2}+e^{-2\alpha_1-\alpha_2}\cO^{s_2s_1s_2}\\
         &\quad +e^{-2\alpha_1-\alpha_2}q_2\cO^{s_1}-e^{-2\alpha_1-\alpha_2}q_2\cO^{s_2s_1}\\
     % \end{align*}}
     %\newpage
  %    \resizebox{.99\hsize}{!}{\parbox{8cm}{\begin{align*}
                   %    &\quad +e^{-2\alpha_1-\alpha_2}q_2\cO^{s_1}-e^{-2\alpha_1-\alpha_2}q_2\cO^{s_2s_1}\\
   \cO^{s_2}\circ \cO^{s_2s_1}&\equiv(1-e^{-\alpha_2})\cO^{s_2s_1}+(1+e^{-\alpha_1})e^{-\alpha_2}\cO^{s_1s_2s_1}\\
                       &\quad +e^{-\alpha_2}\cO^{s_2s_1s_2}-(1+e^{-\alpha_1})e^{-\alpha_2}\cO^{s_1s_2s_1s_2}\\
     \cO^{s_2}\circ \cO^{s_1s_2s_1}&\equiv(1-e^{-2\alpha_1-\alpha_2})\cO^{s_1s_2s_1}+e^{-2\alpha_1-\alpha_2}\cO^{s_1s_2s_1s_2}\\
     \cO^{s_2}\circ \cO^{s_2s_1s_2}&\equiv(1-e^{-2\alpha_1-2\alpha_2})\cO^{s_2s_1s_2}+e^{-2\alpha_1-2\alpha_2}q_2\cO^{s_2s_1}\\
     \cO^{s_2}\circ \cO^{s_1s_2s_1s_2}&\equiv(1-e^{-2\alpha_1-2\alpha_2})\cO^{s_1s_2s_1s_2}+e^{-2\alpha_1-2\alpha_2}q_2\cO^{s_1s_2s_1}\\
           \cO^{s_1s_2}\circ \cO^{s_1s_2}&\equiv(1-e^{-\alpha_1})(1-e^{-2\alpha_1-\alpha_2})\cO^{s_1s_2}+(1-e^{-2\alpha_1-\alpha_2})e^{-\alpha_1}\cO^{s_2s_1s_2}\\ &\quad +e^{-3\alpha_1-\alpha_2}q_2\cO^{s_2s_1}+(1-e^{-\alpha_1})e^{-2\alpha_1-\alpha_2}q_2\cO^{s_1}\\
%  \end{align*}}}
     %\newpage
     %\resizebox{.99\hsize}{!}{\parbox{8cm}{\begin{align*}
      \cO^{s_1s_2}\circ \cO^{s_2s_1}&\equiv(1-e^{-2\alpha_1-\alpha_2})\cO^{s_1s_2s_1}+(1-e^{-\alpha_1-\alpha_2})\cO^{s_2s_1s_2}\\
                         &\quad -(1-e^{-\alpha_1-\alpha_2}-e^{-2\alpha_1-\alpha_2})\cO^{s_1s_2s_1s_2}\\
   \cO^{s_1s_2}\circ  \cO^{s_1s_2s_1}&\equiv(1-e^{-\alpha_1})(1-e^{-2\alpha_1-\alpha_2})\cO^{s_1s_2s_1}
                          +(1-e^{-2\alpha_1-\alpha_2})e^{-\alpha_1}\cO^{s_1s_2s_1s_2}\\
 \cO^{s_1s_2}\circ  \cO^{s_2s_1s_2}&\equiv(1-e^{-\alpha_1-\alpha_2})(1-e^{-2\alpha_1-\alpha_2})\cO^{s_2s_1s_2}\\
                                   &\quad  +e^{-\alpha_1-\alpha_2}q_2\cO^{s_1s_2s_1}+(1-e^{-\alpha_1-\alpha_2})e^{-2\alpha_1-\alpha_2}q_2\cO^{s_2s_1}\\
 \cO^{s_1s_2}\circ  \cO^{s_1s_2s_1s_2}&\equiv(1-e^{-\alpha_1-\alpha_2})(1-e^{-2\alpha_1-\alpha_2})\cO^{s_1s_2s_1s_2}\\
                                      &\quad +(1-e^{-2\alpha_1-\alpha_2})e^{-\alpha_1-\alpha_2}q_2\cO^{s_1s_2s_1}\\
  \cO^{s_2s_1}\circ \cO^{s_2s_1}&\equiv(1-e^{-\alpha_2})(1-e^{-\alpha_1-\alpha_2})\cO^{s_2s_1}+(1-e^{-\alpha_1-\alpha_2})(1+e^{-\alpha_1})e^{-\alpha_2}\cO^{s_1s_2s_1}\\
             &\quad -e^{-\alpha_1-\alpha_2}q_1\cO^{s_2s_1s_2}+(1+e^{-\alpha_1})e^{-\alpha_1-2\alpha_2}q_1\cO^{s_1s_2}+(1-e^{-\alpha_2})e^{-\alpha_1-\alpha_2}q_1\cO^{s_2}\\
    \cO^{s_2s_1}\circ  \cO^{s_1s_2s_1}&\equiv(1-e^{-\alpha_1-\alpha_2})(1-e^{-2\alpha_1-\alpha_2})\cO^{s_1s_2s_1}\\
               &\quad +e^{-2\alpha_1-\alpha_2}q_1\cO^{s_2s_1s_2}+(1-e^{-2\alpha_1-\alpha_2})e^{-\alpha_1-\alpha_2}q_1\cO^{s_1s_2}\\
 \cO^{s_2s_1}\circ  \cO^{s_2s_1s_2}&\equiv(1-e^{-\alpha_2})(1-e^{-\alpha_1-\alpha_2})\cO^{s_2s_1s_2}\\
                                 &\quad +(1-e^{-\alpha_1-\alpha_2})(1+e^{-\alpha_1})e^{-\alpha_2}\cO^{s_1s_2s_1s_2}\\
 \cO^{s_2s_1}\circ  \cO^{s_1s_2s_1s_2}&\equiv(1-e^{-\alpha_1-\alpha_2})(1-e^{-2\alpha_1-\alpha_2})\cO^{s_1s_2s_1s_2}\\
                                 &\quad +(1-e^{-\alpha_1-\alpha_2})(1+e^{-\alpha_1})e^{-\alpha_1-\alpha_2}q_1\cO^{s_2s_1s_2}\\
 \cO^{s_1s_2s_1}\circ  \cO^{s_1s_2s_1}&\equiv(1-e^{-\alpha_1})(1-e^{-\alpha_1-\alpha_2})(1-e^{-2\alpha_1-\alpha_2})\cO^{s_1s_2s_1}\\
                       &\quad +(1-e^{-2\alpha_1-\alpha_2})e^{-\alpha_1}q_1\cO^{s_2s_1s_2}\\
                       &\quad        + (1-e^{-\alpha_1})(1-e^{-2\alpha_1-\alpha_2})e^{-\alpha_1-\alpha_2}q_1\cO^{s_1s_2}\\
 \cO^{s_1s_2s_1}\circ  \cO^{s_2s_1s_2}&\equiv(1-e^{-\alpha_1-\alpha_2})(1-e^{-2\alpha_1-\alpha_2})\cO^{s_1s_2s_1s_2}\\
  \end{align*}}}
     \newpage
     \resizebox{.99\hsize}{!}{\parbox{8cm}{\begin{align*}
 \cO^{s_1s_2s_1}\circ  \cO^{s_1s_2s_1s_2}&\equiv(1-e^{-\alpha_1})(1-e^{-\alpha_1-\alpha_2})(1-e^{-2\alpha_1-\alpha_2})\cO^{s_1s_2s_1s_2}\\
                       &\quad +(1-e^{-\alpha_1-\alpha_2})(1-e^{-2\alpha_1-\alpha_2})e^{-\alpha_1}q_1\cO^{s_2s_1s_2}\\
  %\end{align*}}}
     %\newpage
     %\resizebox{.99\hsize}{!}{\parbox{8cm}{\begin{align*}
                        \cO^{s_2s_1s_2}\circ  \cO^{s_2s_1s_2}&\equiv(1-e^{-\alpha_2})(1-e^{-\alpha_1-\alpha_2})(1-e^{-2\alpha_1-\alpha_2})\cO^{s_2s_1s_2}\\
                       &\quad +(1-e^{-\alpha_1-\alpha_2})(1+e^{-\alpha_1})e^{-\alpha_2}q_2\cO^{s_1s_2s_1}\\
                          &\quad     + (1-e^{-\alpha_2})(1-e^{-\alpha_1-\alpha_2})e^{-2\alpha_1-\alpha_2}q_2\cO^{s_2s_1}\\
 \cO^{s_2s_1s_2}\circ  \cO^{s_1s_2s_1s_2}&\equiv(1-e^{-\alpha_2})(1-e^{-\alpha_1-\alpha_2})(1-e^{-2\alpha_1-\alpha_2})\cO^{s_1s_2s_1s_2}\\
                       &\quad +(1-e^{-\alpha_1-\alpha_2})(1-e^{-2\alpha_1-\alpha_2})e^{-\alpha_2}q_2\cO^{s_1s_2s_1}\\
 \cO^{s_1s_2s_1s_2}\circ  \cO^{s_1s_2s_1s_2}&\equiv(1-e^{-\alpha_1})(1-e^{-\alpha_2})(1-e^{-\alpha_1-\alpha_2})(1-e^{-2\alpha_1-\alpha_2})\cO^{s_1s_2s_1s_2}\\
                       &\quad +(1-e^{-\alpha_2})(1-e^{-\alpha_1-\alpha_2})(1-e^{-2\alpha_1-\alpha_2})e^{-\alpha_1}q_1\cO^{s_2s_1s_2}\\
                       &\quad        + (1-e^{-\alpha_1})(1-e^{-\alpha_1-\alpha_2})(1-e^{-2\alpha_1-\alpha_2})e^{-\alpha_2}q_2\cO^{s_1s_2s_1}\\
  \end{align*}}}

\end{document}